\documentclass[a4paper, 12pt]{amsart}
\usepackage[top=35truemm,bottom=35truemm,left=35truemm,right=35truemm]{geometry}

\usepackage{amscd, amsmath, amssymb, amsthm, graphics}
\usepackage{color} 

\numberwithin{equation}{section}

\theoremstyle{plain}
\newtheorem{thm}{Theorem}[section]
\newtheorem{prop}[thm]{Proposition}
\newtheorem{cor}[thm]{Corollary}
\newtheorem{lem}[thm]{Lemma}

\theoremstyle{definition}
\newtheorem*{defi}{Definition}
\newtheorem{example}[thm]{Example}
\newtheorem{rem}[thm]{Remark}

\newtheorem*{problem}{Problem}

\begin{document}

\title[Austere and arid properties for PF submanifolds]{Austere and arid properties for \\ PF submanifolds in Hilbert spaces}
\author[M. Morimoto]{Masahiro Morimoto}


\address[M. Morimoto]{Department of Mathematics, Graduate School of Science, Osaka City University, 3-3-138 Sugimoto, Sumyoshi-ku, Osaka 558-8585 JAPAN.}
\email{mmasahiro0408@gmail.com}



\thanks{The author was partly supported by the Grant-in-Aid for JSPS Research Fellow (No.18J14857) and by Osaka City University Advanced
Mathematical Institute (MEXT Joint Usage/Research Center on Mathematics and Theoretical Physics).}

\maketitle

\begin{abstract}
Austere submanifolds and arid  submanifolds constitute respectively two different classes of minimal submanifolds in finite dimensional Riemannian manifolds. In this paper we introduce these two notions into a class of proper Fredholm (PF)  submanifolds in Hilbert spaces, discuss their relation and show examples of infinite dimensional austere PF submanifolds and arid PF submanifolds in Hilbert spaces. We also mention a classification problem of minimal orbits in hyperpolar  PF actions on Hilbert spaces.
\end{abstract}

\section*{Introduction}
An \emph{austere submanifold} is a minimal submanifold of a Riemannian manifold which has a local symmetry. More precisely a submanifold $M$ immersed in a finite dimensional Riemannian manifold $\bar{M}$ is called \emph{austere} if for each normal vector $\xi$ the set of eigenvalues with multiplicities of the shape operator $A_\xi$ is invariant under the multiplication by $(-1)$. This notion was originally introduced by Harvey and Lawson \cite{HL82} in the study of calibrated geometry. Except for the case of surfaces the austere condition is much stronger than the minimal one. It is an interesting problem to classify austere submanifolds under suitable conditions (e.g.\ \cite{Bry91}, \cite{DF01}, \cite{II10}, \cite{IST09}, \cite{KM19}).

In \cite{IST09} Ikawa Sakai and Tasaki  introduced a certain kind of austere submanifold which has a global symmetry, which they call a \emph{weakly reflective submanifold}. A 
submanifold $M$ immersed in a finite dimensional Riemannian manifold $\bar{M}$ is called \emph{weakly reflective} if for each normal vector $\xi$ at each $p \in M$ there exists an isometry $\nu_\xi$ of $\bar{M}$ which satisfies
\begin{equation*}
\nu_\xi(p)= p, \quad 
d \nu_\xi(\xi)= - \xi, \quad
\nu_\xi(M) = M.
\end{equation*}
Here we call $\nu_\xi$ a reflection with respect to $\xi$. A reflective submanifold (\cite{Leu73}),  defined as a connected component of the fixed point set of an involutive isometry on $\bar{M}$, is an example of a weakly reflective submanifold. Another  example is a singular orbit of a cohomogeneity one action, which was essentially shown to be weakly reflective by Podest\`a \cite{Pod97}. It is an interesting problem to study submanifold geometry of orbits under  isometric actions of Lie groups and to determine their weakly reflective orbits (e.g.\ \cite{IST09}, \cite{Ohno16}, \cite{Eno18}).

Recently Taketomi \cite{Tak18} introduced a generalized concept of weakly reflective submanifolds, namely \emph{arid submanifolds}. A submanifold $M$ immersed in a finite dimensional Riemannian manifold $\bar{M}$ is called \emph{arid} if for each nonzero normal vector $\xi$ at each $p \in M$ there exists an isometry $\varphi_\xi$ of $\bar{M}$ which satisfies 
\begin{equation*}
\varphi_\xi(p) = p
, \quad
d \varphi_\xi(\xi) \neq \xi
, \quad
\varphi_\xi(M) = M.
\end{equation*}
Here we call $\varphi_\xi$ an isometry with respect to $\xi$. From this definition we have:
\begin{equation*}
\begin{array}{ccccccc}
&&&&\text{ austere }&& \vspace{-3mm}
\\
&&&\rotatebox{30}{$\hspace{-1.8mm}\Rightarrow$}&&\rotatebox{-30}{$\hspace{-1.8mm}\Rightarrow$}&\vspace{-2.5mm}
\\
\text{reflective}&\Rightarrow&\text{weakly reflective }&&&&\text{minimal}. \vspace{-2.5mm}
\\
&&&\rotatebox{-30}{$\hspace{-1.8mm}\Rightarrow$}&&\rotatebox{30}{$\hspace{-1.8mm}\Rightarrow$}& \vspace{-3.5mm}
\\
&&&&\text{ arid }&&
\end{array}
\end{equation*}
In \cite{Tak18} he gave an example of an arid submanifold which is \emph{not} an austere submanifold (therefore \emph{not} a weakly reflective submanifold). Also he showed that any isolated orbit of a proper isometric action is an arid submanifold. It is a problem for a given proper isometric action to determine their arid orbits which are not isolated. 

It is also interesting to study these submanifolds and problems in the infinite dimensional case.

A fundamental class of infinite dimensional submanifolds is given by \emph{proper Fredholm} (PF) submanifolds in Hilbert spaces (Terng \cite{Ter89}), where the shape operators are compact operators and the distance functions satisfy the Palais-Smale condition (\cite{Pal63}, \cite{Sma64}). 
It is known that many examples of PF submanifolds are obtained through a certain Riemannian submersion $\Phi_K: V_\mathfrak{g} \rightarrow G/K$ which is called the \emph{parallel transport map}  (\cite{KT93}, \cite{TT95}). Here $G/K$ is a compact normal homogeneous space and $V_\mathfrak{g} := L^2([0,1], \mathfrak{g})$ the Hilbert space of all $L^2$-paths with values in the Lie algebra $\mathfrak{g}$ of $G$ (cf.\ Section \ref{Sec2}). It was proved (\cite{TT95}) that if $N$ is a closed submanifold of $G/K$ then its inverse image $\Phi_K^{-1}(N)$ is a PF submanfold of $V_\mathfrak{g}$. A fundamental problem is to study the geometrical relation between $N$ and $\Phi_K^{-1}(N)$. For example, it was shown (\cite{KT93}, \cite{Koi02}, \cite{HLO06}) that if $N$ is minimal then $\Phi_K^{-1}(N)$ is also minimal (in some sense).

Recently the author \cite{M19}  introduced the concept of weakly reflective submanifolds into a class of PF submanifolds in Hilbert spaces and studied the geometrical relation between submanifolds $N$ and $\Phi_K^{-1}(N)$. To simplify explanation here we suppose that $G/K$ is an irreducible Riemannian symmetric space of compact type. The author showed (\cite[Theorem 8]{M19}): 
\begin{enumerate}
\item[] \emph{If $N$ is weakly reflective, then $\Phi_K^{-1}(N)$ is  also weakly reflective.}
\end{enumerate}
It is noted that even if $N$ is reflective, $\Phi_K^{-1}(N)$ can \emph{not} be reflective (\cite[Remark 5]{M19}). Then there are two questions:
\begin{enumerate}
\item[] Question A: \quad  \emph{If $N$ is austere, is $\Phi_K^{-1}(N)$ austere ?}
\item[] Question B:  \emph{\quad If $N$ is arid, is $\Phi_K^{-1}(N)$ arid ?} 
\end{enumerate}

The purpose of this paper is to give  answers to Questions A and B as far as possible. Main results are Theorem \ref{sphere2} and Theorem \ref{thm4}.  Here Theorem \ref{sphere2} gives an affirmative answer to Question A under the assumption that  $G/K$ is a sphere, while in other cases it remains unsolved because there is no simple relation between the principal curvatures of $N$ and $\Phi_K^{-1}(N)$. 
On the other hand Theorem \ref{thm4} is just an affirmative answer to Question B. Note that Theorem \ref{thm4} is proved similarly to the weakly reflective case \cite{M19}. As an application we show examples of arid PF submanifolds which are \emph{not} austere (therefore \emph{not} weakly reflective) PF submanifolds (Examples \ref{example3}, \ref{example1}, \ref{example2}). Also we mention a classification problem of minimal orbits in the \emph{$P(G,H)$-action}, which is an isometric PF action of a Hilbert Lie group $P(G,H)$ on the Hilbert space $V_\mathfrak{g}$.

This paper is organized as follows. In Section 1 we define the arid property  for PF submanifolds in Hilbert spaces. In Section 2 we prepare the setting of $P(G,H)$-actions and the parallel transport map. In Section 3 we calculate  the principal curvatures of PF submanifolds obtained through the parallel transport map. In Section 4, to answer Question A, we study the austere property of so obtained PF submanifolds. In section 5,  to answer  Question B, we study the arid property of so obtained PF submanifolds.  In Section 6 we mention the classification problem of minimal orbits in hyperpolar $P(G,H)$-actions.

\section{PF submanifolds with symmetries}
Let $V$ be a separable Hilbert space over $\mathbb{R}$ and $M$ a Hilbert manifold immersed in $V$. The end point map $Y: T^\perp M \rightarrow V$ is defined by $Y(\xi) := p + \xi$ for $\xi \in T^\perp_p M$. $M$ is called \emph{proper Fredholm} (PF) (\cite{Ter89}) if it has finite codimension and the restriction of $Y$ to a normal disc bundle of any finite radius is proper and Fredholm. It is known (\cite[Proposition 2.7]{Ter89}) that for each $p \in M$ and each $\xi \in T^\perp _p M$ the shape operator $A_\xi$ is a self-adjoint compact operator on $T_pM$ which is not of trace class in general. 

At present there are three kinds of definitions for `minimal' PF submanifolds, namely $\zeta$-minimal (\cite{KT93}), $r$-minimal (\cite{HLO06}) and $f$-minimal (\cite{Koi02}). Here we recall the first two which are used in this paper. Let $M$ be a PF submanifold of $V$ and $\xi \in T^\perp M$. We denote by 
\begin{equation*}
\mu_1 \leq \mu_2 \leq \cdots 
<0<
\cdots \leq \lambda_2 \leq \lambda_1.
\end{equation*}
the eigenvalues repeated with multiplicities of the shape operator $A_\xi$. 
 
We say that $A_\xi$ is \emph{$\zeta$-regularizable} (\cite{KT93}) if 
$
\sum_{k} \lambda_k^s + \sum_{k} |\mu_k|^s < \infty
$ for all $s >1$ and  
\begin{equation*}
\operatorname{tr}_{\zeta} A_\xi := \lim_{s \searrow1}\left(\sum_{k} \lambda_k^s - \sum_{k} |\mu_k|^s\right)
\end{equation*} 
exists. Then we call $\operatorname{tr}_\zeta A_\xi$ the \emph{$\zeta$-regularized mean curvature}  in the direction of $\xi$.  $M$ is called \emph{$\zeta$-regularizable} if $A_\xi$ is $\zeta$-regularizable for all $\xi \in T^\perp M$. If $M$ is $\zeta$-regularizable and $\operatorname{tr}_{\zeta} A_\xi$ vanishes for all $\xi \in T^\perp M$,  we say that $M$ is \emph{$\zeta$-minimal}. 

We say that $A_\xi$ is \emph{regularizable} (\cite{HLO06}) if $\operatorname{tr} A_\xi^2 < \infty$ and 
\begin{equation*}
\operatorname{tr}_{r} A_\xi := \sum_{k=1}^\infty (\lambda_k + \mu_k)
\end{equation*} 
converges, where we regard $\lambda_k$ or $\mu_k$ as zero if there are less than $k$ positive or negative eigenvalues, respectively. Then we call 
$
\operatorname{tr}_{r} A_\xi
$
the \emph{regularized mean curvature} in the direction of $\xi$.  $M$ is called  \emph{regularizable} if $A_\xi$ is  regularizable for all $\xi \in T^\perp M$. If $M$ is regularizable and $\operatorname{tr}_{r} A_\xi$ vanishes for all $\xi \in T^\perp M$,  we say that $M$ is \emph{r-minimal}.

In \cite{M19} the concepts of reflective submanifolds, weakly reflective submanifolds and austere submanifolds were introduced into a class of PF submanifolds in Hilbert spaces. Similarly we can define arid PF submanifolds:
\begin{defi}
Let $M$ be a PF submanifold of $V$. $M$ is called \emph{arid} if for each $p \in M$ and each $\xi \in T^\perp_p M \backslash \{0\}$ there exists an isometry $\varphi_\xi$ of $V$ which satisfies 
\begin{equation*}
\varphi_\xi(p) = p
, \quad 
d \varphi_\xi(\xi) \neq \xi
, \quad 
\varphi_\xi(M) = M.
\end{equation*}
\end{defi}
We have the following relation: 
\begin{equation*}
\begin{array}{ccccccc}
&&&&\text{austere PF}&& \vspace{-3mm}
\\
&&&\rotatebox{30}{$\hspace{-1.8mm}\Rightarrow$}&&&\vspace{-2.5mm}
\\
\text{reflective PF}&\Rightarrow&\text{weakly reflective PF}&&&&\text{} \vspace{-2.5mm}
\\
&&&\rotatebox{-30}{$\hspace{-1.8mm}\Rightarrow$}&&& \vspace{-3.5mm}
\\
&&&&\text{\ arid PF.}&&
\end{array}
\end{equation*}
The $\zeta$-minimality and $r$-minimality of austere PF submanifolds were discussed in \cite[Section 1]{M19}. We do not know in general whether arid PF submanifolds are $\zeta$-minimal or $r$-minimal because it is not clear that the mean curvature vector of a PF submanifold is well-defined or not. However it will not interfere our purpose because we will give attention to PF submanifolds obtained through the parallel transport map (\cite{KT93}, \cite{TT95}), where the mean curvature vector is well-defined (\cite[Theorem 4.12]{KT93}, \cite[Lemma 5.2]{HLO06}) and thus the arid property implies both $\zeta$-minimality and $r$-minimality.

\section{$P(G,H)$-actions and the parallel transport map} \label{Sec2}

In this section we prepare the setting of $P(G, H)$-actions and the parallel transport map. The related and detailed facts can be found in \cite[Section 2]{M19} and references therein. 

Let $G$ be a connected compact Lie group with Lie algebra $\mathfrak{g}$. Choose an $\operatorname{Ad}(G)$-invariant inner product of $\mathfrak{g}$ and equip the corresponding bi-invariant Riemannian metric with $G$. Denote by $\mathcal{G} := H^1([0,1], G)$ the Hilbert Lie group of all Sobolev $H^1$-paths in $G$ parametrized on $[0,1]$ and by $V_\mathfrak{g} := H^0([0,1], \mathfrak{g})$ the Hilbert space of all Sobolev $H^0$-paths in $\mathfrak{g}$ parametrized on $[0,1]$. For each $a \in G$ (resp.\ $x \in \mathfrak{g}$) we denote by $\hat{a} \in \mathcal{G}$ (resp.\ $\hat{x} \in V_\mathfrak{g}$)  the constant path which values at $a$ (resp.\ $x$). Then $\mathcal{G}$ acts on $V_\mathfrak{g}$  via the gauge transformations:
\begin{equation}\label{gaugetransf}
g* u := gug^{-1} - g' g^{-1}
, \quad
g \in \mathcal{G}, \ u \in V_\mathfrak{g},
\end{equation}
where $g'$ denotes the weak derivative of $g$ with respect to the parameter on $[0,1]$. We know that this action is isometric, transitive, proper and Fredholm (\cite[Theorem 5.8.1]{PT88}).

Let $H$ be a closed subgroup of $G \times G$ with Lie algebra $\mathfrak{h}$. Define a Lie subgroup $P(G,H)$ of $\mathcal{G}$ by 
\begin{equation*}
P(G, H) := \{g \in \mathcal{G} \mid (g(0), g(1)) \in H\}
\end{equation*}
with Lie algebra
\begin{equation*}
\operatorname{Lie}P(G,H)=
\{Z \in H^1([0,1], \mathfrak{g}) \mid  (Z(0), Z(1)) \in \mathfrak{h}\}.
\end{equation*}
The induced action of $P(G,H)$ on $V_\mathfrak{g}$ is called the \emph{$P(G,H)$-action}. We know that $P(G,H)$-action is isometric, proper and Fredholm (\cite[p.\ 132]{Ter95}). Thus each orbit of the $P(G,H)$-action is a PF submanifold of $V_\mathfrak{g}$ (\cite[Theorem 7.1.6]{PT88}). Also we know (\cite[Corollary 4.2]{TT95}) that both $P(G, \{e\} \times G)$-action on $V_\mathfrak{g}$ and $P(G, G \times \{e\})$-action on $V_\mathfrak{g}$ are simply transitive.


The \emph{parallel transport map} (\cite{KT93}, \cite{Ter95}) $\Phi: V_\mathfrak{g} \rightarrow G$ is a Riemannian submersion defined by 
\begin{equation*}
\Phi(u) :=E_u(1), \quad u \in V_\mathfrak{g},
\end{equation*}
where $E_u \in \mathcal{G}$ is the unique solution to the linear ordinary differential equation
\begin{equation*}
\left\{
\begin{array}{l}
E_u^{-1} E_u = u, 
\\
E_u(0) = e.
\end{array}
\right.
\end{equation*}
Note that $G \times G$ acts on $G$ isometrically by 
\begin{equation} \label{Haction}
(b_1, b_2) \cdot a := b_1 a b_2^{-1}
,\quad
a, b_1, b_2 \in G.
\end{equation}
We know (\cite[Proposition 1.1]{Ter95}) that for $g \in \mathcal{G}$ and $u \in V_\mathfrak{g}$
\begin{equation} \label{equiv4}
\Phi(g * u) = (g(0), g(1)) \cdot \Phi(u)
\end{equation}
and that for any closed subgroup $H$ of $G \times G$, 
\begin{equation}\label{iimage}
P(G, H) * u = \Phi^{-1}(H \cdot \Phi(u)).
\end{equation}
More generally, if $N$ is a closed submanifold of $G$ then the inverse image $\Phi^{-1}(N)$ is a PF submanifold of $V_\mathfrak{g}$ (\cite[Lemma 5.1]{TT95}). 

The differential $(d \Phi)_{\hat{0}} : T_{\hat{0}} V_\mathfrak{g} \rightarrow T_e G \cong \mathfrak{g}$ of $\Phi$ at $\hat{0} \in V_\mathfrak{g}$ (cf.\ \cite[p.\ 685]{TT95}) is given by 
\begin{equation*}
(d \Phi)_{\hat{0}} (X) = \int_0^1 X(t) dt
, \quad
X \in T_{\hat{0}} V_\mathfrak{g} \cong V_\mathfrak{g}.
\end{equation*}
 Using this we obtain the orthogonal direct sum decomposition 
\begin{equation*}
T_{\hat{0}} V_\mathfrak{g} 
= 
\hat{\mathfrak{g}} 
\oplus 
\operatorname{Ker} (d\Phi)_{\hat{0}}, 
\quad
X = \left(\int_0^1 X(t) dt \right) \oplus \left( X - \int_0^1 X(t) dt\right).
\end{equation*}

Let $K$ be a closed subgroup of $G$ with Lie algebra $\mathfrak{k}$. Denote by $\mathfrak{g}= \mathfrak{k} + \mathfrak{m}$ the orthogonal direct sum  decomposition. Restricting the $\operatorname{Ad}(G)$-invariant inner product of $\mathfrak{g}$ to $\mathfrak{m}$ we  define the induced $G$-invariant Riemannian metric on the  homogeneous space $G/K$. Such a  metric is called a normal homogeneous metric and $G/K$ a compact normal homogeneous space. We write $\pi : G \rightarrow G/K$ for the natural projection, which is a Riemannian submersion with totally geodesic fiber. The \emph{parallel transport map $\Phi_K$ over $G/K$} is a Riemannian submersion defined by 
\begin{equation} \label{ptm2} 
\Phi_K := \pi \circ \Phi : V_\mathfrak{g} \rightarrow G \rightarrow G/K.
\end{equation}

Note that for each $g \in P(G, G \times \{e\})$ and $a \in G$ the diagrams
\begin{align} \label{commute1}
\text{(i)}\quad 
\begin{CD}
V_\mathfrak{g}@>g* >> V_\mathfrak{g}
\\
@V\Phi VV @V\Phi VV
\\
G @> (g(0),\, e) >> G
\end{CD}
\quad \quad \quad \quad \quad 
\text{(ii)} \quad 
\begin{CD}
G @>l_a>> G
\\
@V\pi VV @V\pi VV
\\
G/K @>L_a>> G/K
\end{CD}
\end{align}
commute, where $l_a$ denotes the left translation by $a$ and $L_a$ the  isometry on $G/K$ defined by $L_a(bK) := (ab)K$. Thus setting $a := g(0)$ we obtain a commutative diagram
\begin{equation} \label{commute3}
\begin{CD}
V_\mathfrak{g} @>g*>> V_\mathfrak{g}
\\
@V \Phi_K VV  @V\Phi_{K} VV
\\
\ G/K\  @> L_a >> \ G/K.
\end{CD}
\end{equation}



Let $G$, $K$ be as above. Denote by $F_K := \Phi_K^{-1}(eK) = \Phi^{-1}(K)$ the fiber of $\Phi_K$ at $eK \in G/K$ and by $F := \Phi^{-1}(e)$ the fiber of $\Phi$ at $e \in G$. Since 
\begin{equation*}
F_K = \Phi^{-1}((\{e\} \times K)\cdot e) = P(G, \{e\} \times K) * \hat{0}
\end{equation*}
we can write (cf.\ \cite[equation (9)]{M19})
\begin{equation} \label{fiber1}
T_{\hat{0}} F_K = \{-Z' \mid Z \in H^1([0,1], \mathfrak{g}), \ Z(0)=0, \ Z(1) \in \mathfrak{k}\}.
\end{equation}
Similarly we have
\begin{equation}\label{fiber2}
T_{\hat{0}} F = \{-Q' \mid Q \in H^1([0,1], \mathfrak{g}), \ Q(0)=Q(1) =0 \}.
\end{equation}

Suppose that a closed submanifold $N$ of $G/K$ through $eK \in G/K$ is given. Since $\Phi$ and $\Phi_K$ are Riemannian  submersions we have the orthogonal direct sum decompositions
\begin{equation*}
T_{\hat{0}} \Phi_K^{-1}(N)
\cong
T_{\hat{0}} F_K \oplus T_{eK} N
\cong 
T_{\hat{0}} F \oplus \mathfrak{k} \oplus T_{eK} N.
\end{equation*}
From now on $A^{\Phi_K^{-1}(N)}$, $A^{\pi^{-1}(N)}$ and $A^N$ denote the shape operators of $\Phi_K^{-1}(N)$, $\pi^{-1}(N)$ and $N$, respectively. We fix a normal vector $\xi$ of $N$ at $eK$. Then its horizontal lift at $\hat{0} \in \Phi_K^{-1}(N)$ is given by the constant path $\hat{\xi}$. 
\begin{prop}
Suppose $[\mathfrak{m}, \mathfrak{m}] \subset \mathfrak{k}$. For $-Q' \in T_{\hat{0}} F$, $x \in \mathfrak{k}$, $y \in T_{eK} N$
\begin{align}
\label{so1}
&A^{\Phi_K^{-1}(N)}_{\hat{\xi}}(-Q')
=
[\hat{\xi}, Q] - \left[\xi, \int_0^1Q(t) dt\right]^\perp,
\\
\label{so2}
&A^{\Phi_K^{-1}(N)}_{\hat{\xi}}(x)
=
\frac{1}{2}[\xi, x]^\perp - t[\xi, x],
\\
\label{so3}
&A^{\Phi_K^{-1}(N)}_{\hat{\xi}}(y)
=
A^N_\xi(y) + 
(1- t )[\xi, y],
\end{align}
where $\perp$ denote the projection from $\mathfrak{g}$ onto $T^\perp_{eK} N (\subset \mathfrak{m})$.
\end{prop}
\begin{proof}
The formula (\ref{so1}) follows from  \cite[Corollary 1 (ii)]{M19}. Also from  \cite[Corollary 1 (i)]{M19} we have the following formula: for $v \in T_{e} \pi^{-1}(N)$
\begin{equation*}
A^{\Phi_K^{-1}(N)}_{\hat{\xi}}(v)
=
A^{\pi^{-1}(N)}_\xi(v) - \left(t - \frac{1}{2}\right)[\xi, v].
\end{equation*}
Further by \cite[Remark 1 (ii)]{M19}, for $v = x \oplus y \in  T_{e} \pi^{-1}(N) = \mathfrak{k} \oplus T_{eK} N$ 
\begin{equation}\label{so7}
A^{\pi^{-1}(N)}_\xi(v)
=
A^N_\xi(y) + \frac{1}{2}[\xi, y]_{\mathfrak{k}} - \frac{1}{2}[\xi, x]^\top,
\end{equation}
where the subscript $\mathfrak{k}$ and $\top$ denote the projections onto $\mathfrak{k}$ and $T_e N$ respectively.   Since we are supposing $[\mathfrak{m}, \mathfrak{m}] \subset \mathfrak{k}$ these formulas imply (\ref{so2}) and (\ref{so3}).
\end{proof}
\begin{cor} For $-Z' \in T_{\hat{0}} F_K$
\begin{equation}\label{so4}
A^{\Phi_K^{-1}(N)}_{\hat{\xi}}
(-Z')
=
[\hat{\xi}, Z]- \left[\xi, \int_0^1 Z(t) dt\right]^\perp.
\end{equation}
\end{cor}
\begin{proof}
Set $Q := Z - t Z(1)$ and $x := -Z(1)$ so that $-Z' = -Q' + x$. By (\ref{so1}) and (\ref{so2}) the desired formula follows.
\end{proof}

\section{Principal curvatures}
In this section we calculate the principal curvatures of PF submanifolds obtained through the parallel transport map. For technical reasons, here we will restrict our attention to PF submanifolds obtained from \emph{curvature adapted} submanifolds in compact \emph{symmetric} spaces. Although such a subject has been studied by Koike \cite{Koi02} there are some inaccuracies in his eigenspace decomposition and so here we give the corrected formula with another elementary proof by using the formulas for shape operators prepared in the last section.

Recall that a submanifold $M$ immersed in a Riemannian manifold $\bar{M}$ is called \emph{curvature adapted} if for each $p \in M$ and each $\xi \in T^\perp_p M$ the Jacobi operator $R_\xi := \bar{R}(\cdot, \xi)\xi: T_p \bar{M} \rightarrow T_p \bar{M}$, where $\bar{R}$ denotes the curvature tensor of $\bar{M}$, satisfies 
\begin{equation*}
R_\xi(T_p M) \subset T_p M
\quad \text{and} \quad 
A^M_\xi \circ R_\xi|_{T_p M} = R_\xi|_{T_p M} \circ A^M_\xi,
\end{equation*}
where $A^M_\xi$ denotes the shape operator of $M$ in the direction of $\xi$. 

Let $G$ be a connected compact Lie group and $K$ a closed subgroup of $G$.  Suppose that $K$ is symmetric, that is, there exists an involutive automorphism $\theta$ of $G$ such that 
$
G^\theta_0 \subset K \subset G^\theta
$, 
where $G^\theta$ is the fixed point subgroup of $\theta$ and $G^\theta_0$ the identity component. Denote by $\mathfrak{g}$ and $\mathfrak{k}$ the Lie algebras of $G$ and $K$ respectively and by   
$
\mathfrak{g} = \mathfrak{k} + \mathfrak{m}
$
the direct sum decomposition into the $\pm 1$-eigenspaces of $d \theta$. We 
fix an inner product $\langle \cdot, \cdot \rangle$ of $\mathfrak{g}$ which is invariant under both $Ad(G)$ and $\theta$. Then the above direct sum  decomposition is orthogonal with respect to this inner product $\langle \cdot, \cdot \rangle$. We equip the corresponding bi-invariant Riemannian metric with $G$ and a normal homogeneous metric with $G/K$. Then $G/K$ is a compact symmetric space and the natural projection $\pi : G \rightarrow G/K$ is a Riemannian submersion with totally geodesic fiber. 
We denote by $\Phi_K: V_\mathfrak{g} \rightarrow G/K$ the parallel transport map.

Let $N$ be a curvature adapted  closed submanifold of $G/K$. Note that in order to calculate the principal curvatures of a PF submanifold $\Phi_K^{-1}(N)$ we can assume without loss of generality that $N$ contains $e K$ and moreover it suffices to consider normal vectors only at $\hat{0} \in \Phi_K^{-1}(N)$ because of the commutativity (\ref{commute3}).  Thus in the rest of this section we fix $\xi \in T^\perp_{eK} N$ and calculate the principal curvatures of $\Phi_K^{-1}(N)$ in the direction of $\hat{\xi} \in T^\perp_{\hat{0}} \Phi_K^{-1}(N)$. Note that in this case the Jacobi operator  is given by $R_\xi = - \operatorname{ad}(\xi)^2 : \mathfrak{m} \rightarrow \mathfrak{m}$.

Denote by $\{\sqrt{-1}\, \nu\}$ the set of all distinct eigenvalues of the skew adjoint operator $\operatorname{ad}(\xi): \mathfrak{g} \rightarrow \mathfrak{g}$. Consider the complexification $\operatorname{ad}(\xi): \mathfrak{g}^\mathbb{C} \rightarrow \mathfrak{g}^\mathbb{C}$ and the eigenspace decomposition
\begin{equation*}
\mathfrak{g}^\mathbb{C}
=
\mathfrak{g}_0^\mathbb{C}
+
\sum_{\nu \neq 0} \mathfrak{g}_\nu,
\end{equation*}
\begin{align*}
\mathfrak{g}_0
&:= 
\{x \in \mathfrak{g} \mid \operatorname{ad}(\xi)(x)=0 \},
\\
\mathfrak{g}_\nu
&:=
\{z \in \mathfrak{g}^\mathbb{C} \mid \operatorname{ad}(\xi)(z) = \sqrt{-1}\, \nu z\}.
\end{align*}
Since $\bar{\mathfrak{g}}_{\nu} = \mathfrak{g}_{- \nu}$ we can write
\begin{equation*}
\mathfrak{g}^\mathbb{C}
=
\mathfrak{g}_0^\mathbb{C}
+
\sum_{\nu >0 } (\mathfrak{g}_\nu + \mathfrak{g}_{- \nu})
\end{equation*}
and thus we obtain
\begin{equation*}
\mathfrak{g} = \mathfrak{g}_0 + \sum_{\nu>0} (\mathfrak{g}_\nu + \mathfrak{g}_{- \nu})_\mathbb{R},
\end{equation*}
\begin{equation*}
(\mathfrak{g}_\nu + \mathfrak{g}_{- \nu})_\mathbb{R}
=
\{x \in \mathfrak{g} \mid
\operatorname{ad}(\xi)^2(x)
= - \nu^2 x \},
\end{equation*}
which is nothing but the eigenspace decomposition with respect to $\operatorname{ad}(\xi)^2 : \mathfrak{g} \rightarrow \mathfrak{g}$. 
Since $\operatorname{ad}(\xi)^2$ commutes with involution $\theta$ we have the simultaneous eigenspace  decomposition
\begin{equation}\label{esd1}
\mathfrak{k}
= 
\mathfrak{k}_0
+
\sum_{\nu > 0} \mathfrak{k}_{\nu}
, \qquad
\mathfrak{m} = \mathfrak{m}_0
 + \sum_{\nu > 0}
\mathfrak{m}_{\nu},
\end{equation}
\begin{align*}
\mathfrak{k}_0
 &:= \{x \in \mathfrak{k} \mid \operatorname{ad}(\xi)(x) = 0\},
\\
\mathfrak{m}_0
 &:= \{y \in \mathfrak{m} \mid \operatorname{ad}(\xi)(y) = 0\},
\\
\mathfrak{k}_{\nu} 
&:=
\{x \in \mathfrak{k} \mid \operatorname{ad}(\xi)^2 (x) = - \nu^2 x\},
\\
\mathfrak{m}_{\nu} 
&:=
\{y \in \mathfrak{m} \mid \operatorname{ad}(\xi)^2 (y) = - \nu^2 y\}.
\end{align*}
By similar arguments as in \cite[p.\ 60]{Loo69II}, for each $\nu>0$ we can take bases $\{x_1^\nu, \cdots, x_{m(\nu)}^\nu\}$ of $\mathfrak{k}_\nu$ and $\{y_1^\nu, \cdots, y_{m(\nu)}^\nu\}$ of $\mathfrak{m}_\nu$ where $m(\nu) := \dim \mathfrak{k}_\nu = \dim \mathfrak{m}_\nu$ such that 
\begin{equation}\label{basis0}
[\xi, x^\nu_k] = - \nu y^\nu_k
, \qquad
[\xi, y^\nu_k] = \nu x^\nu_k.
\end{equation}
Thus a linear isomorphism $\varphi_\nu: \mathfrak{k}_\nu \rightarrow \mathfrak{m}_\nu$ is defined by 
\begin{equation}\label{isom}
\varphi_\nu(x) := - \frac{1}{\nu}[\xi, x].
\end{equation}

Let $\{\lambda\}$ denote the set of all distinct eigenvalues of the shape operator $A^N_\xi$. 
Set 
\begin{equation*}
S_\lambda := \operatorname{Ker}(A^N_\xi - \lambda \operatorname{id}).
\end{equation*}
Since $N$ is curvature adapted, for each $\nu \geq 0$ we have  the decomposition
\begin{equation*}
\begin{array}{ccccc}
\mathfrak{m} &=& T_{eK} N &\oplus& T^\perp_{eK} N \vspace{1mm}
\\
\cup&&\cup&&\cup
\\
\mathfrak{m}_\nu &=&  \mathfrak{m}_\nu \cap T_{eK} N  &\oplus&  \mathfrak{m}_\nu \cap T^\perp_{eK} N \vspace{1mm}
\\
&&\rotatebox{90}{$=$}&&
\\
&&\sum_{\lambda} ( \mathfrak{m}_\nu \cap S_\lambda ).&& 
\end{array}
\end{equation*}
For each $\nu \geq 0$ we set
\begin{equation*}
m(\nu,\lambda) :=\dim (\mathfrak{m}_\nu \cap S_\lambda)
, \qquad
m(\nu,\perp) :=\dim (\mathfrak{m}_\nu \cap T^\perp_{eK} N).
\end{equation*}
For each $\nu \geq 0$ and $\lambda$, choose 
\begin{align*}
&\{y^{(\nu, \lambda)}_1, \cdots , y^{(\nu, \lambda)}_{m(\nu, \lambda)}\}: \ \text{a basis of $\mathfrak{m}_\nu \cap S_\lambda$},
\\
&\{y^{(\nu, \perp)}_1, \cdots , y^{(\nu, \perp)}_{m(\nu, \perp)}\}:\  \text{a basis of $\mathfrak{m}_\nu \cap T^\perp_{eK} N$.}
\end{align*}
Then for each $\nu \geq 0$ we obtain
\begin{equation*}
\{y^{(\nu, \lambda)}_1, \cdots , y^{(\nu, \lambda)}_{m(\nu, \lambda)} \}_\lambda \cup \{y^{(\nu, \perp)}_1, \cdots , y^{(\nu, \perp)}_{m(\nu, \perp)}\}
: \ 
\text{a basis of $\mathfrak{m}_\nu$}.
\end{equation*} 
Thus for each $\nu > 0$, via an isomorphism (\ref{isom}) we obtain
\begin{equation*}
\{x^{(\nu, \lambda)}_1, \cdots , x^{(\nu, \lambda)}_{m(\nu, \lambda)} \}_\lambda \cup \{x^{(\nu, \perp)}_1, \cdots , x^{(\nu, \perp)}_{m(\nu, \perp)}\}
: \ 
\text{a basis of $\mathfrak{k}_\nu$.}
\end{equation*}
For $\nu = 0$  we choose and denote by 
\begin{equation*}
\{x_1^0, \cdots, x^0_{\dim \mathfrak{k}_0
}\}
: \ 
\text{a basis of $\mathfrak{k}_0
$}.
\end{equation*}
Note that these satisfy 
\begin{equation*}
\begin{array}{lcl}
\ [\xi, x^0_i] =0, 
&& 
[\xi, y^{(0, \lambda)}_j] = [\xi, y^{(0, \perp)}_l]= 0,
\\
\ [\xi, x_k^{(\nu, \lambda)}] = - \nu\,  y_k^{(\nu, \lambda)},
&&
\ [\xi, y_k^{(\nu, \lambda)}] = \nu\,  x_k^{(\nu, \lambda)},
\\
\ [\xi, x_r^{(\nu, \perp)}] = - \nu\, y_r^{(\nu, \perp)},
&& 
\ [\xi, y_r^{(\nu, \perp)}] = \nu\,  x_r^{(\nu, \perp)}.
\end{array}
\end{equation*}

Set $V(\mathfrak{g}) := V_\mathfrak{g} = H^0([0,1], \mathfrak{g})$. We decompose
\begin{equation*}
V(\mathfrak{g}) 
=
\sum_{\nu \geq 0} V(\mathfrak{k}_\nu) 
+ 
\sum_{\nu \geq 0}
\left(V(\mathfrak{m}_\nu \cap T_{eK} N) 
+ 
V(\mathfrak{m}_\nu \cap T_{eK}^\perp N)\right)
\end{equation*}
and equip a basis with each term above. Recall that there are well-known three kinds of orthonormal bases in $H^0([0,1], \mathbb{R})$:
\begin{align}
\label{basis1}&\{
1
, \ 
\sqrt{2}\, \cos 2n \pi t
, \ 
\sqrt{2}\, \sin 2 n \pi t
\}_{n = 1}^\infty \ ,
\\
\label{basis2}&
\{
1
, \ 
\sqrt{2}\, \cos n \pi t
\}_{n = 1}^\infty\ , 
\\
\label{basis3}&\{
\sqrt{2}\, \sin n \pi t
\}_{n = 1}^\infty\ .
\end{align}
For $\nu = 0$ we consider the following bases:
\begin{equation*}
\begin{array}{lcccc}
\text{a basis of $V(\mathfrak{k}_0
)$}&: && \hspace{-20mm}\{x_i^0 \sin n \pi t\}_{i, \, n} \ ,\hspace{-25mm}&
\\
\text{a basis of $V(\mathfrak{m}_0
 \cap T_{eK} N)$}&: &\{y_j^{(0, \lambda)}\}_{\lambda, \, j}&\cup&\{y_j^{(0, \lambda)} \cos n \pi t\}_{\lambda, \, j, \, n}\ ,
\\
\text{a basis of $V(\mathfrak{m}_0
 \cap T^\perp_{eK} N)$}&: &\{y^{(0, \perp)}_l\}_l& \cup &\{y^{(0, \perp)}_l \cos n \pi t\}_{l, \, n}\ .
\end{array}
\end{equation*}
For each $\nu >0$ we consider the following bases:
\begin{equation*}
\begin{array}{lcccc}
\text{a basis of $V(\mathfrak{k}_\nu)$}&: &
\{x_k^{(\nu, \lambda)} \sin n \pi t\}_{\lambda, \, k, \,n} 
&\cup& \{x^{(\nu, \perp)}_{r} \sin n \pi t\}_{r, \, n}\ ,
\\
\text{a basis of $V(\mathfrak{m}_\nu \cap T_{eK} N)$}&: &\{y^{(\nu, \lambda)}_k\}_{\lambda, \, k} & \cup&\{y^{(\nu, \lambda)}_k \cos n \pi t\}_{\lambda, \, n,\,k}\ ,
\\
\text{a basis of $V(\mathfrak{m}_\nu \cap T^\perp_{eK} N)$}&: &\{y^{(\nu, \perp)}_r\}_r&\cup &\{y^{(\nu, \perp)}_r \cos n \pi t\}_{n,\,r} \ .
\end{array}
\end{equation*}
Clearly all these bases form a basis of $V(\mathfrak{g}) = V_\mathfrak{g}$. Identifying $T_{\hat{0}} V_\mathfrak{g} \cong V_\mathfrak{g}$  and considering the orthogonal direct sum decomposition
\begin{equation*}
T_{\hat{0}} V_\mathfrak{g}
=
T_{\hat{0}} \Phi_K^{-1}(N)
\oplus
T^\perp_{eK} N 
, \qquad
X 
=
\left(X - \int_0^1 X(t)^\perp dt\right)
\oplus
\int_0^1 X(t)^\perp dt
\end{equation*}
we consequently obtain the following basis of $T_{\hat{0}} \Phi_K^{-1}(N)$: 
\begin{align*}
&
\{x_i^0 \sin n \pi t\}_{i, \, n}
\cup
\{y_j^{(0, \lambda)}\}_{\lambda, \,j} 
\cup
\{y^{(0, \lambda)}_j \cos n \pi t\}_{\lambda, \,j, \, n}
\cup
\{y^{(0, \perp)}_r \cos n \pi t\}_{r, \, n}
\\
& 
\ \cup \  
\bigcup_{\nu > 0}
\left(
\{x_k^{(\nu, \lambda)} \sin n \pi t\}_{\lambda,\,  k, \,n}
\cup
\{y^{(\nu, \lambda)}_k\}_{\lambda, \, k} 
\cup
\{y^{(\nu, \lambda)}_k \cos n \pi t\}_{\lambda, \, k, \, n}
\right)
\\
&
\ \cup \ 
\bigcup_{\nu > 0}
\left(
\{x^{(\nu, \perp)}_r \sin n \pi t\}_{r,n}
\cup
\{y^{(\nu, \perp)}_r \cos n \pi t\}_{r, \, n} 
\right).
\end{align*}

\begin{lem}\label{lem4.3}
\begin{align*}
&\textup{(i)}&&
A^{\Phi_K^{-1}(N)}_{\hat{\xi}}(x^0_i \sin n \pi t) = 0, \qquad 
A^{\Phi_K^{-1}(N)}_{\hat{\xi}}(y^{(0, \lambda)}_j)
=
\lambda y^{(0, \lambda)}_j,&&
\\
&\textup{(ii)}&&
A^{\Phi_K^{-1}(N)}_{\hat{\xi}}(y^{(0, \lambda)}_j \cos n \pi t)
= 
A^{\Phi_K^{-1}(N)}_{\hat{\xi}}(y^{(0, \perp)}_l \cos n \pi t)
=0, &&
\\
&\textup{(iii)}& &
A^{\Phi_K^{-1}(N)}_{\hat{\xi}}
(x_r^{(\nu, \perp)} \sin n \pi t)
=
- \frac{\nu}{n \pi} y^{(\nu, \perp)}_r  \cos n \pi t,&&
\\
&&&
A^{\Phi_K^{-1}(N)}_{\hat{\xi}}
(y_r^{(\nu, \perp)} \cos n \pi t)
= 
- \frac{\nu}{n \pi} x^{(\nu, \perp)}_r \sin n \pi t,&&
\\
&\textup{(iv)}&&A^{\Phi_K^{-1}(N)}_{\hat{\xi}}(y^{(\nu, \lambda)}_k)
=
\lambda y^{(\nu, \lambda)}_k
+
\frac{2 \nu}{\pi} \sum_{n = 1}^\infty \frac{1}{n} (x^{(\nu, \lambda)}_k\sin n \pi t ) ,&&
\\
&\textup{(v)}&&A^{\Phi_K^{-1}(N)}_{\hat{\xi}}(x^{(\nu, \lambda)}_k \sin n \pi t)
=
- \frac{\nu}{n \pi } y^{(\nu, \lambda)}_k (-1 + \cos n \pi t) ,&&
\\
&\textup{(vi)}&&A^{\Phi_K^{-1}(N)}_{\hat{\xi}}(y^{(\nu, \lambda)}_k \cos n \pi t)
=
- \frac{\nu}{n \pi} x^{(\nu, \lambda)}_k \sin n \pi t .&&
\end{align*}
\end{lem}
\begin{proof}
(i) and (ii): The second equality of (i) follows from \eqref{so3}. Let $Q$ be 
\begin{equation*}
\frac{1}{n \pi}
x_i^0 \cos n \pi t
, \quad
- \frac{1}{n \pi}
y^{(0, \lambda)}_j  \sin n \pi t
\quad \text{or} \quad
-\frac{1}{n \pi}
y^{(0, \perp)}_l \sin n \pi t.
\end{equation*}
By \eqref{so1} we obtain the first formula of (i) and formulas in (ii). (iii): Set
\begin{equation*}
Z_1 := \frac{1}{n \pi }x^{(\nu, \perp)}_r (-1 + \cos n \pi t)
, \quad
Z_2 := - \frac{1}{n \pi } y^{(\nu, \perp)}_r \sin n \pi t. 
\end{equation*}
Then we have
\begin{equation*}
[\xi, Z_1] 
=
- \frac{\nu}{n\pi } y^{(\nu, \perp)}_r (-1 + \cos n \pi t)
, \quad
[\xi, Z_2]
= 
- \frac{\nu}{n\pi} x^{(\nu, \perp)}_r \sin n \pi t.
\end{equation*}
Thus we have 
\begin{equation*}
\left[\xi, \int_0^1 Z_1(t) dt\right]^\perp
=
\frac{\nu}{n \pi} y^{(\nu, \perp)}_r 
, \quad
\left[\xi, \int_0^1 Z_2(t) dt\right]^\perp =0.
\end{equation*}
Hence by \eqref{so4} the desired equalities follow. (iv): By (\ref{so3}) we have
\begin{equation*}
A^{\Phi_K^{-1}(N)}_{\hat{\xi}}(y^{(\nu, \lambda)}_k)
=
\lambda y^{(\nu, \lambda)}_k
+
\left(1 - t \right) 
\nu
x^{(\nu, \lambda)}_k.
\end{equation*}
Since
$
\int_0^1 (1-t) \sin n \pi t \, dt = (n \pi)^{-1}
$ the Fourier expansion with respect to a basis \eqref{basis3} of a function 
$f: [0,1] \rightarrow \mathbb{R}$, $t \mapsto 1-t$  is given by
\begin{equation*}
f 
=
\frac{2}{\pi} \sum_{n = 1}^\infty \frac{1}{n} (\sin n \pi t ).
\end{equation*}
This shows the desired equality. (v), (vi): Set 
\begin{equation*}
Z_1 := \frac{1}{ n \pi }x_k^{(\nu, \lambda)}  (-1 + \cos  n \pi t)
, \quad 
Z_2 := - \frac{1}{ n \pi } y_k^{(\nu, \lambda)} \sin  n \pi t.
\end{equation*}
Then we have 
\begin{equation*}
[\hat{\xi}, Z_1]
=
- \frac{\nu}{ n \pi} y_k^{(\nu, \lambda)}  (-1 + \cos  n \pi t)
, \quad
[\hat{\xi}, Z_2]
=
-  \frac{\nu}{ n \pi } x_k^{(\nu, \lambda)} \sin  n \pi t,
\end{equation*}
Thus we have 
\begin{equation*}
\left[\xi, \int_0^1 Z_1(t) dt\right]^\perp 
= 
\left[\xi, \int_0^1 Z_2(t) dt \right]^\perp 
= 0.
\end{equation*}
Hence by \eqref{so4} we obtain the desired formulas.
\end{proof}

We come now to the principal curvatures of a PF submanifold $\Phi_K^{-1}(N)$: 
\begin{thm}\label{pc1}
Let $G/K$ be a compact symmetric space, $\Phi_K: V_\mathfrak{g} \rightarrow G/K$ the parallel transport map, $N$ a curvature adapted closed submanifold of $G/K$ through $e K \in G/K$, and $\xi \in T^\perp _{eK} N \subset \mathfrak{m}$. Denote by $\{\sqrt{-1}\, \nu\}$ the set of all distinct eigenvalues of $\operatorname{ad}(\xi): \mathfrak{g} \rightarrow \mathfrak{g}$
and by $\{\lambda\}$ the set of all  distinct eigenvalues of the shape operator $A^N_\xi$. For each $\nu>0$, each $\lambda$  and each $m \in \mathbb{Z}$ we set
\begin{equation*}
\mu 
= 
\mu(\nu, \lambda, m)
:=
\frac{\nu}{\arctan \frac{\nu}{\lambda}+ m \pi},
\end{equation*}
where we set $\arctan (\nu/\lambda) := \pi/2$ if $\lambda = 0$. Then the principal curvatures of a PF submanifold $\Phi^{-1}_K(N)$ in the direction of $\hat{\xi} \in T^\perp_{\hat{0}} \Phi^{-1}(N)$ are given by 
\begin{equation*}
\left\{0, \ \lambda, \ 
\frac{\nu}{n \pi}, \ 
\mu(\nu, \lambda, m)
\right\}_{\lambda, \ \nu > 0, \ n \in \mathbb{Z}\backslash \{0\}, \  m \in \mathbb{Z}}.
\end{equation*}
The eigenfunctions and the multiplicities are given in the following table.
\begin{table}[h]
\renewcommand{\arraystretch}{2.0}
\begin{tabular}{|c|c|c|} \hline
eigenvalue & basis of eigenfunctions & multiplicity
\\ \hline 
$0$ & $
\{x_i^0 \sin n \pi t, 
y^{(0, \lambda)}_j \cos n \pi t,  y^{(0, \perp)}_l \cos n \pi t \}_{n \in \mathbb{Z}_{\geq 1}, \, \lambda,\, i, \, j, \, l}
$& $\infty$  
\\ \hline 
$\lambda$ & $\{y^{(0, \lambda)}_j\}_j$
& $m(0, \lambda)$  
\\ \hline
${\displaystyle \frac{\nu}{n \pi}}$& $
\{x^{(\nu, \perp)}_r \sin n \pi t - y^{(\nu, \perp)}_r \cos n \pi t
\}_r
$ & $m(\nu, \perp)$
\\ \hline
$\mu(\nu, \lambda, m)$& $
\left\{
\sum_{n \in \mathbb{Z}}
\frac{\nu}{n \pi \mu+ \nu} 
(x^{(\nu, \lambda)}_k \sin n \pi t
+
y^{(\nu, \lambda)}_k \cos n \pi t)
\right\}_{k}$ & $m(\nu, \lambda)$
\\ \hline 
\end{tabular}
\end{table}
\end{thm}
\noindent
\begin{rem}
Let $\mathfrak{a}$ be a maximal abelian subspace of $\mathfrak{m}$ and $\Delta^+$ the set of positive roots satisfying $\alpha(\xi) \geq 0$ for each $\alpha \in \Delta^+$. Then for each $\nu>0$ there exists $\alpha \in \Delta^+$ such that $\nu = \alpha(\xi)$ and thus the above eigenvalues coincide with those given by Koike \cite[Theorem 3.3]{Koi02}. 
However note that the eigenspace  decomposition \cite[p.\ 73, line 3]{Koi02} does not hold in general.
\end{rem}
\begin{proof}[Proof of Theorem  \ref{pc1}] 
By Lemma \ref{lem4.3} (i) - (iii) it follows that $0$, $\lambda$ and  $\frac{\nu}{n \pi}$ are eigenvalues of $A^{\Phi_K^{-1}(N)}_{\hat{\xi}}$ with eigenfunctions described  above. Let $W$ denote a subspace of $T_{\hat{0}}\Phi_K^{-1}(N)$ spanned by all such eigenfunctions and consider its orthogonal complements $W^\perp$ in $T_{\hat{0}}\Phi_K^{-1}(N)$. We know that one basis of $W^\perp$ is given by 
\begin{equation*}
\bigcup_{\nu>0}
\left(
\{y^{(\nu, \lambda)}_k\}_{\lambda, \, k} 
\cup
\{x_k^{(\nu, \lambda)} \sin n \pi t
, \ 
y^{(\nu, \lambda)}_k \cos n \pi t\}_{\lambda, \ k, \ n \in \mathbb{Z}_{\geq 1}}
\right).
\end{equation*}
In particular Lemma \ref{lem4.3} (iv) - (vi) show that for each $\nu> 0$, $\lambda$ and $k$, a subspace of $T_{\hat{0}}\Phi_K^{-1}(N)$ spanned by 
\begin{equation*}
\{y^{(\nu, \lambda)}_k\} 
\cup
\{
x_k^{(\nu, \lambda)} \sin n \pi t
, \ 
y^{(\nu, \lambda)}_k \cos n \pi t
\}_{n \in \mathbb{Z}_{\geq 1}}
\end{equation*}
is invariant  under $A^{\Phi_K^{-1}(N)}_{\hat{\xi}}$. We denote such a subspace by $W^\perp_{(\lambda, \nu, k)}$. Suppose that for constants $a_n, b_n, c \in \mathbb{R}$
\begin{equation*}
\varphi :=
c y^{(\nu, \lambda)}_k
+
\sum_{n = 1}^\infty 
\{
a_n (x^{(\nu, \lambda)}_k \sin n \pi t)
+
b_n (y^{(\nu, \lambda)}_k \cos n \pi t)
\}
\quad \in W^\perp_{(\lambda,\, \nu, \, k)}
\end{equation*}
is a (nonzero) eigenfunction of $A^{\Phi_K(N)}_{\hat{\xi}}$ for some eigenvalue $\mu$. 
By Lemma \ref{lem4.3} (iv) - (vi) we have
\begin{align*}
A^{\Phi_K^{-1}(N)}_{\hat{\xi}}(\varphi)
&=
\left(
c \lambda 
+
\frac{\nu}{\pi}
\sum_{n=1}^\infty
\frac{a_n}{n}
\right) 
y^{(\nu, \lambda)}_k
\\
&\quad 
+
\frac{\nu}{\pi}
\sum_{n=1}^\infty
\frac{2c - b_n}{n} 
(x^{(\nu, \lambda)}_k \sin n \pi t)
-
\frac{\nu}{\pi}
\sum_{n=1}^\infty
\frac{a_n}{n} 
(y^{(\nu, \lambda)}_k \cos n \pi t).
\end{align*}
Comparing with 
\begin{equation*}
\mu \varphi 
=
\mu c y^{(\nu, \lambda)}_k
+
\sum_{n = 1}^\infty 
\{
\mu a_n (x^{(\nu, \lambda)}_k \sin n \pi t)
+
\mu  b_n (y^{(\nu, \lambda)}_k \cos n \pi t)
\}
\end{equation*}
we obtain a system of equations
\begin{align}
& \label{req1}
c \lambda 
+
\frac{\nu}{\pi}
\sum_{n=1}^\infty
\frac{a_n}{n}
=
c \mu,
\\
&\label{req2}
\frac{\nu}{\pi}
\frac{2 c - b_n}{n}
=
\mu a_n,
\\
& \label{req3}
-
\frac{\nu}{\pi}
\frac{a_n}{n} 
=
\mu b_n .
\end{align}

Summing (\ref{req3}) with respect to $n \in \mathbb{Z}_{\geq 1}$ we have 
\begin{equation*}
- \frac{\nu}{\pi}
 \sum_{n = 1}^\infty \frac{a_n}{n} 
= 
\mu 
\sum_{n = 1}^\infty
b_n.
\end{equation*}
Applying this to (\ref{req1}) we obtain \begin{equation}\label{key1}
c \lambda - \mu \sum_{n =1}^\infty b_n 
=
c \mu.
\end{equation}

On the other hand, multiplying (\ref{req3}) by $\mu$ we have 
\begin{equation*}
- \frac{\nu}{n \pi} \mu a_n 
=
\mu^2 b_n .
\end{equation*}
Applying (\ref{req2}) to this we obtain
\begin{equation*}
-2 c \left(\frac{\nu}{n \pi} \right)^2 
=
b_n \left(
\mu^2 - \left(\frac{\nu}{n \pi} \right)^2 
\right)
\end{equation*}
for \emph{all} $n \in \mathbb{Z}_{\geq 1}$. Note that this implies $c \neq 0$, $\mu \neq 0$ and 
$\mu ^2 - (\frac{\nu}{n \pi})^2 \neq 0$
for all $n \in \mathbb{Z}$. Therefore without loss of generality we can (and will) assume $c = 1$ from now on, and 
\begin{equation}\label{aaa}
b_n 
=
\frac{-2 r^2}{n^2 -r^2}
, \quad
\text{where}
\quad r := \frac{\nu}{\pi \mu}.
\end{equation}

From (\ref{key1}), (\ref{aaa}) and the standard formula (see \cite[1.449 - 4]{GR07}): 
\begin{equation*}
\sum_{n = 1}^\infty \frac{1}{n^2 - a^2}
=
\frac{1}{2a^2} - \frac{\pi}{2a} \cot(\pi a)
, \quad a \in \mathbb{R} \backslash \mathbb{Z}
\end{equation*}
we have
\begin{align*}
\lambda - \mu
&=
\mu \sum_{n = 1}^\infty b_n
=
- \mu + \nu \cot (\nu/\mu).
\end{align*}
Thus  
\begin{equation*}
\mu
=
\frac{\nu}{\arctan (\nu/\lambda)+ m \pi}
, \quad
m \in \mathbb{Z}.
\end{equation*}

By \eqref{aaa} we have
\begin{align*}
b_n 
&=
\frac{-2 \nu^2}{(n \pi \mu)^2 - \nu^2}
=
\frac{-2(\arctan(\nu/\lambda) + m \pi)^2}
{(n \pi)^2 - (\arctan(\nu/\lambda) + m \pi)^2}
\\
&=
(\arctan(\nu/\lambda) + m \pi) 
\\
& \quad  \times 
\left(
\frac{1}{n\pi + (\arctan(\nu/\lambda) + m \pi)}
-
\frac{1}{n\pi - (\arctan(\nu/\lambda) + m \pi)}
\right).
\end{align*}
By \eqref{req3} we have
\begin{align*}
a_n 
&=
- n \pi b_n \mu \nu^{-1}
=
\frac{2n \pi (\arctan(\nu/\lambda) + m \pi)}
{(n \pi)^2 - (\arctan(\nu/\lambda) + m \pi)^2}
\\
&= 
(\arctan(\nu/\lambda) + m \pi) 
\\
& \quad  \times 
\left(
\frac{1}{n\pi + (\arctan(\nu/\lambda) + m \pi)}
+
\frac{1}{n\pi - (\arctan(\nu/\lambda) + m \pi)}
\right).
\end{align*}
Therefore we obtain
\begin{align*}
\varphi 
&=
y^{(\nu, \lambda)}_k
+
\sum_{n\in \mathbb{Z} \backslash\{0\}} 
\frac{\arctan(\nu/\lambda) + m \pi}{n\pi + (\arctan(\nu/\lambda) + m \pi)}
\{
(x^{(\nu, \lambda)}_k \sin n \pi t)
+
(y^{(\nu, \lambda)}_k \cos n \pi t)
\}
\\
&=
\sum_{n \in \mathbb{Z}}
\frac{\nu}{n\pi \mu + \nu}
\{(x^{(\nu, \lambda)}_k \sin n \pi t)
+
(y^{(\nu, \lambda)}_k \cos n \pi t)\}.
\end{align*}
This proves the theorem.
\end{proof}
Considering the case that $N = \{eK\}$ we obtain: 
\begin{cor}\label{fiber1}
Let $G/K$ be a compact symmetric space. Take $\xi \in T_{eK} (G/K) = \mathfrak{m}$. Denote by $\{\sqrt{-1}\, \nu\}$ the set of all distinct eigenvalues of $\operatorname{ad}(\xi): \mathfrak{g} \rightarrow \mathfrak{g}$. 
Then the principal curvatures of the fiber $\Phi^{-1}_K(eK)$ in the direction of $\hat{\xi} \in T^\perp_{\hat{0}} \Phi_K^{-1}(eK)$ are given by 
\begin{equation*}
\left\{0,  
\ \frac{\nu}{n \pi}
\right\}_{\nu > 0, \ n \in \mathbb{Z}\backslash \{0\}}.
\end{equation*}
Denoting by $\{x_i^0\}_i$ a basis of $\mathfrak{k}_0
$, by $\{y_j^0\}_j$ a basis of $\mathfrak{m}_0
$ and by $\{x^\nu_k\}_k$, $\{y^\nu_k\}_k$ the bases defined by \eqref{basis0}, the eigenfunctions and the multiplicities are given in the following table.
\begin{table}[h]
\renewcommand{\arraystretch}{2.0}
\begin{tabular}{|c|c|c|} \hline
eigenvalue & basis of eigenfunctions & multiplicity 
\\ \hline 
$0$ & $
\{x_i^{0} \sin n \pi t,  y^{0}_j \cos n \pi t \}_{n \in \mathbb{Z}_{\geq 1}, \, \lambda,\, i, \, j}
$& $\infty$  
\\ \hline 
${\displaystyle \frac{\nu}{n \pi}}$& $
\{x^{\nu}_k \sin n \pi t - y^{\nu}_k \cos n \pi t
\}_k
$ & $m(\nu)$
\\ \hline 
\end{tabular}
\end{table}
\end{cor}


The principal curvatures of $\pi^{-1}(N)$ are  given by the following proposition, which is proved by the straightforward computations using the formula (\ref{so7}). 
\begin{prop}\label{pc2}
With notation as in Theorem \ref{pc1}, for each $\nu>0$ and each $\lambda$ we set
\begin{align*}
\kappa_+ 
&=
\kappa_+(\nu, \lambda)
:=
\frac{1}{2}(\lambda + \sqrt{\lambda^2 + \nu^2}),
\\
\kappa_-
&=
\kappa_-(\nu, \lambda)
:=
\frac{1}{2}(\lambda - \sqrt{\lambda^2 + \nu^2}).
\end{align*}
Then the principal curvatures of $\pi^{-1}(N)$ in the direction of $\xi \in T^\perp_e \pi^{-1}(N) \cong T^\perp _{eK} N$ are given by  
\begin{equation*}
\left\{
0, \ \lambda, \ \kappa_+(\nu, \lambda)
, \ \kappa_-(\nu, \lambda)
\right\}_{\lambda, \  \nu > 0}.
\end{equation*}
The eigenfunctions and the multiplicities are given in the following table.

\begin{table}[h]
\renewcommand{\arraystretch}{2.0}
\begin{tabular}{|c|c|c|} \hline
eigenvalue & basis of eigenfunctions & multiplicity
\\ \hline 
$0$ 
&
$\{x_i^0\}_i, \ \{x^{(\nu, \perp)}_r\}_{r, \, \nu}$
&
$
\dim \mathfrak{k}_0
 + \dim T^\perp_{eK} N
$   
\\ \hline 
$\lambda$ & $\{y^{(0, \lambda)}_j\}_j$
& $m(0, \lambda)$
\\ \hline
$
\kappa_+(\nu, \lambda) $
&$\{
\nu x^{(\nu, \lambda)}_k + 2 \kappa_+ y^{(\nu, \lambda)}_k\}_k$
&
$m(\nu, \lambda)$
\\ \hline
$\kappa_-(\nu, \lambda)$ 
&$\{\nu  x^{(\nu, \lambda)}_k + 2 \kappa_- y^{(\nu, \lambda)}_k\}_k$
&
$m(\nu, \lambda)$
\\ \hline 
\end{tabular}
\end{table}
\end{prop}

\section{The austere property}\label{austerevia}

In this section we study the austere property of PF submanifolds obtained through the parallel transport map $\Phi_K$. Notice that even if $N$ is an austere curvature adapted  submanifold of a compact symmetric space $G/K$, it is not clear in general whether the inverse image $\Phi_K^{-1}(N)$ is austere or not; according to Theorem \ref{pc1} it is not clear whether the set of eigenvalues with \emph{multiplicities} of the shape operator is invariant under the multiplication by $(-1)$ or not. 
From this reason here we will restrict our attention further to the case that $G/K$ is a sphere and show that in this case $N$ is austere if and only if $\Phi_K^{-1}(N)$ is austere (Theorem \ref{sphere2}). 

Let $S^{l}(r) = G/K$ denote the $l$-dimensional sphere of radius $r>0$, where $l \in \mathbb{Z}_{\geq 1}$ and $(G,K) =(SO(l+1), SO(l))$. Note that in this case any submanifold of $G/K$ is automatically curvature adapted. Let $N$ be a closed submanifold of $G/K$. Suppose $eK \in N$ and fix $\xi \in T^\perp_{eK} N$. Then for $v \in T_{eK} N$ and $\eta \in T^\perp_{eK} N$ the Jacobi operator $R_\xi$ satisfies
\begin{equation*}
R_\xi(v) := \frac{\|\xi\|^2}{r^2} v, 
\qquad 
R_\xi(\eta) := 
\frac{1}{r^2} \{\|\xi\|^2\eta - \langle \eta, \xi\rangle \xi \}.
\end{equation*}
Thus in this case the eigenspace decomposition (\ref{esd1}) of $\mathfrak{m}$ is given by 
\begin{equation*}
\mathfrak{m} = \mathfrak{m}_0
 + \mathfrak{m}_{\nu}
, \quad \quad
\text{where}
\quad \nu := \|\xi\|/r,
\end{equation*}
\begin{align*}
&\mathfrak{m}_0
 =  
\mathbb{R} \xi \subset T^\perp_{eK} N,
\\
&
\mathfrak{m}_{\nu} = T_{eK} N \oplus 
\{\eta \in T^\perp _{eK} N \mid \eta \perp \xi\}.
\end{align*}
In particular we have 
\begin{align*}
&\mathfrak{m}_0
 \cap T_{eK} N = \{0\}
, \qquad 
\mathfrak{m}_0
 \cap T^\perp_{eK} N = \mathbb{R}\xi, 
\\
&\mathfrak{m}_\nu \cap T_{eK} N = T_{eK} N
, \quad
\mathfrak{m}_\nu \cap T^\perp_{eK} N = \{\eta \in T^\perp _{eK} N \mid \eta \perp \xi\}.
\end{align*}
Hence by Proposition \ref{pc2} the principal curvatures of $\pi^{-1}(N)$ in the direction of $\xi \in T^\perp_e \pi^{-1}(N) \cong T^\perp_{eK} N$ are given by
\begin{equation}\label{eigen4.1}
\left\{
0, \ \kappa_+(\|\xi\|/r, \lambda)
, \ \kappa_-(\|\xi\|/r, \lambda)
\right\}_{\lambda}.
\end{equation}
Further by Theorem \ref{pc1} the principal curvatures of a PF submanifold $\Phi^{-1}_K(N)$ in the direction of $\hat{\xi} \in T^\perp_{\hat{0}} \Phi^{-1}(N)$ are given by
\begin{equation}\label{eigen4.2}
\left\{0, \ 
\frac{\|\xi\|}{r n \pi}, \ 
\mu(\|\xi\|/r, \lambda, m)
\right\}_{\lambda, \ n \in \mathbb{Z}\backslash \{0\}, \  m \in \mathbb{Z}}.
\end{equation}
Notice that in this case the multiplicities of 
\begin{equation*}
\lambda
,\quad
\kappa_+(\|\xi\|/r, \lambda)
, \quad
\kappa_-(\|\xi\|/r, \lambda)
, \quad
\mu(\|\xi\|/r, \lambda, m)
\end{equation*}
are the same for each $\lambda$.
\begin{thm}\label{sphere2}
Let $N$ be a closed submanifold of the $l$-dimensional sphere $S^{l}(r) = G/K$ of radius $r>0$, where $l \in \mathbb{Z}_{\geq 1}$ and $(G,K) =(SO(l+1), SO(l))$. Then the the following are equivalent:
\begin{enumerate}
\item[(i)] $N$ is an austere submanifold of $G/K$, 
\item[(ii)] $\pi^{-1}(N)$ is an austere submanifold of $G$,
\item[(iii)] $\Phi_K^{-1}(N)$ is an austere PF submanifold of $V_\mathfrak{g}$.
\end{enumerate}
\end{thm}
\begin{proof}
``(i) $\Rightarrow$ (ii)'': Take $a \in \pi^{-1}(N)$ and $ w \in T_a^\perp \pi^{-1}(N)$. Set $\eta:= d \pi (w) \in T^\perp_{aK} N$, $N' := L_a^{-1}(N)$ and $\xi := d L_a^{-1}(\eta) \in T^\perp_{eK}N'$. Denote by $v \in T_e^\perp \pi^{-1}(N')$ the horizontal lift of $\xi$. By commutativity of (\ref{commute1}) (ii), we have $l_a (\pi^{-1}(N')) = \pi^{-1}(N)$ and $d l_a (v) = w$. Thus in order to show the austerity of $A^{\pi^{-1}(N)}_w$ it suffices to show that of $A^{\pi^{-1}(N')}_{v}$. For each eigenvalue $\lambda$ of $A^{N'}_\xi$ it follows from the austerity of $A^{N'}_\xi$ that $-\lambda$ is also an eigenvalue of $A^{N'}_\xi$ and 
\begin{align*}
&(-1) \times \kappa_+(\|\xi\|/r, \lambda)
=
\kappa_-(\|\xi\|/r, - \lambda),
\\
&(-1) \times \kappa_-(\|\xi\|/r, \lambda)
=
\kappa_+(\|\xi\|/r, - \lambda).
\end{align*}
Note that these identities still hold even if the multiplicities are taking account of. This shows that the set \eqref{eigen4.1} with multiplicities is invariant under the multiplication by $(-1)$ and (ii) follows.

``(ii) $\Rightarrow$ (i)'': Take $aK \in N$ and $\eta \in T_{aK}^\perp N$. Denote by $w \in T^\perp_a \pi^{-1}(N)$ the horizontal lift of $\eta$. Defining $N'$, $\xi$, $v$ by the above way it suffices to show the austerity of $A^{N'}_\xi$. Let $\lambda$ be an eigenvalue of $A^{N'}_\xi$. Since the set \eqref{eigen4.1} is invariant under the multiplication by $(-1)$ there exist eigenvalues $\lambda'$ and $\lambda''$ of $A^{N'}_\xi$ such that 
\begin{align*}
&
(-1) \times \kappa_+(\|\xi\|/r, \lambda)
=
\kappa_-(\|\xi\|/r, \lambda'),
\\
&
(-1) \times \kappa_-(\|\xi\|/r, \lambda)
=
\kappa_+(\|\xi\|/r, \lambda'').
\end{align*}
Note that the function
$\mathbb{R} \rightarrow \mathbb{R}_{>0}$, $x \mapsto \kappa_+(\|\xi\|/r, x)$ 
is monotonically increasing. Thus the relation 
\begin{equation*}
\kappa_-(\|\xi\|/r, x) 
=
- \kappa_+ (\|\xi\|/r, - x)
\end{equation*}
shows that also the function
$\mathbb{R} \rightarrow \mathbb{R}_{<0}$, $x \mapsto \kappa_-(\|\xi\|/r, x)$ is  monotonically increasing. From these we obtain $\lambda' = \lambda'' = - \lambda$. Note that this identity still holds even if the multiplicities are taken account of. This shows (i).

``(i) $\Rightarrow$ (iii)'': Take $u \in \Phi^{-1}_K(N)$ and  $X \in T_u^\perp \Phi_K^{-1}(N)$. Also take $g \in P(G, G \times \{e\})$ so that $u = g * \hat{0}$. Set $a := g(0) = \Phi(u)$, $\eta:= d \Phi_K (X) \in T^\perp_{aK} N$, $N' := L_a^{-1}(N)$ and $\xi := d L_a^{-1}(\eta) \in T^\perp_{eK}N'$. Denote by $\hat{\xi} \in T_{\hat{0}}^\perp \Phi_K^{-1}(N')$ the horizontal lift of $\xi$. By  commutativity of (\ref{commute3}) we have $g* (\Phi_K^{-1}(N')) = \Phi_K^{-1}(N)$ and $d g* ({\hat{\xi}}) = X$. Thus in order to show the austerity of $A^{\Phi_K^{-1}(N)}_X$ it suffices to show that of $A^{\Phi_K^{-1}(N')}_{{\hat{\xi}}}$. For each $\lambda$ it follows from the austerity of $A^{N'}_\xi$ that $- \lambda$ is also an eigenvalue of $A^{N'}_\xi$ and 
\begin{equation*}
(-1) \times \mu(\|\xi\|/r, \lambda, m)
=
\mu(\|\xi\|/r, -\lambda, -m).
\end{equation*}
Note that this identity still hold even if the multiplicities are taken account of.
This shows that the set 
\begin{equation*}
\left\{
\mu(\|\xi\|/r, \lambda, m)
\right\}_{\lambda, \ m \in \mathbb{Z}}
\end{equation*}
with multiplicities is invariant under the multiplication by $(-1)$. This together with \eqref{eigen4.2} shows the austerity of $A^{\Phi_K^{-1}(N')}_{\hat{\xi}}$ and (iii) follows.

``(iii) $\Rightarrow$ (i)'': Take $aK \in N$ and  $\eta \in T_{aK}^\perp N$. Choose $u \in \Phi_K^{-1}(a)$. Denote by $X \in T^\perp_u \Phi_K^{-1}(N)$ the horizontal lift of $\eta$. Defining $N'$, $\xi$, ${\hat{\xi}}$ by the above way it suffices to show the austerity of $A^{N'}_\xi$. From \eqref{eigen4.2} and the assumption the set 
\begin{equation*}
\{\mu(\|\xi\|/r, \lambda, m)\}_{\lambda, \ m \in \mathbb{Z}}
\end{equation*}
with multiplicities is invariant under the multiplication by $(-1)$. Thus for each eigenvalue $\lambda$ of $A^{N'}_\xi$ and each $m \in \mathbb{Z}$ there exists an eigenvalue $\lambda'$ of $A^{N'}_\xi$ and $m' \in \mathbb{Z}$ such that 
\begin{equation*}
(-1) \times \mu(\|\xi\|/r, \lambda, m)
=
\mu(\|\xi\|/r, \lambda', m'). 
\end{equation*}
That is, 
\begin{equation*}
-\arctan \frac{\|\xi\|}{r \lambda} - m \pi
=
\arctan \frac{\|\xi\|}{ r \lambda'} + m' \pi.
\end{equation*}
Since $- \pi/2 <\arctan x < \pi/2$, the above equality shows $m' = -m$ and 
\begin{equation*}
\lambda' = - \lambda.
\end{equation*}
Note that this identity still holds even if  the multiplicities are taking account of.  This shows (i).  
\end{proof}

\begin{example}
Ikawa, Sakai and Tasaki (\cite[Theorem 5.1]{IST09}) classified  austere submanifolds of the standard sphere given as orbits of $s$-representations of irreducible Riemannian symmetric pairs. Applying Theorem \ref{sphere2} to their result we obtain austere PF submanifolds as follows. Let $(U, L)$ be a compact Riemannian symmetric pair, where $L$ is connected. Denote by $\mathfrak{u} = \mathfrak{l} \oplus \mathfrak{p}$ the canonical decomposition and by $\operatorname{Ad}: L \rightarrow SO(\mathfrak{p})$ the isotropy representation. If an orbit $\operatorname{Ad}(L) \cdot x$ through $x \in \mathfrak{p}$ is an austere submanifold of the hypersphere $S(\|x\|)$ in $\mathfrak{p}$ then the orbit $P(SO(\mathfrak{p}), \operatorname{Ad}(L) \times SO(\mathfrak{p})_x) * \hat{0}$ is an austere PF submanifold of the Hilbert space $V_{ \mathfrak{o}(\mathfrak{p})}$. 
\end{example}

\section{The arid property}
In this section we study the arid property of PF submanifolds obtained through the parallel transport map $\Phi_K$. The main theorem is the following:
\begin{thm}\label{thm4}
Let $G$ be a connected compact semisimple Lie group equipped with a bi-invariant Riemannian metric induced from a negative multiple of the Killing form and $K$ a symmetric subgroup of $G$ such that the pair $(G,K)$ effective. If $N$ is an arid submanifold of the symmetric space $G/K$ then
\begin{enumerate}
\item[(i)] $\pi^{-1}(N)$ is an arid submanifold of $G$, and
\item[(ii)] $\Phi_K^{-1}(N)$ is an arid PF submanifold of $V_\mathfrak{g}$.
\end{enumerate}
\end{thm}
From this theorem we obtain the following corollary:
\begin{cor} Let $M$ be an irreducible Riemannian symmetric space of compact type. Denote by $G$ the identity component of the group of isometries of $M$ and by $K$ the isotropy subgroup of $G$ at a fixed $p \in M$. If $N$ is an arid submanifold of $M  =G/K$ then $\Phi_K^{-1}(N)$ is an arid PF submanifold of $V_\mathfrak{g}$.
\end{cor}
To prove Theorem \ref{thm4} we need the following lemma:
\begin{lem} \label{weaksubm}
Let $\mathcal{M}$ and $\mathcal{B}$ be Riemannian Hilbert manifolds, $\phi: \mathcal{M} \rightarrow \mathcal{B}$ a Riemannian submersion and $N$ a closed submanifold of $\mathcal{B}$. Fix $p \in \phi^{-1}(N)$ and $X \in (T^\perp_p \phi^{-1}(N)) \backslash \{0\}$. Suppose that $\varphi_{\mathcal{M}}$ is an isometry of $\mathcal{M}$ fixing $p$, that $\varphi_{\mathcal{B}}$ is an isometry of $\mathcal{B}$ fixing $\phi(p)$ and that the diagram
\begin{equation*}
\begin{CD}
\mathcal{M} @>\varphi_{\mathcal{M}}>> \mathcal{M}
\\
@V\phi VV @V\phi VV
\\
\mathcal{B} @>\varphi_{\mathcal{B}}>> \mathcal{B}
\end{CD}
\end{equation*}
commutes. Then the following are equivalent:
\begin{enumerate}
\item[(i)] $\varphi_\mathcal{M}$ satisfies $\varphi_\mathcal{M}(\phi^{-1}(N)) = \phi^{-1}(N)$ and $d\varphi_{\mathcal{M}}(X) \neq X$.
\item[(ii)] $\varphi_\mathcal{B}$ satisfies $\varphi_\mathcal{B}(N) = N$ and $d \varphi_{\mathcal{B}}(d \phi(X)) \neq d \phi (X)$. 
\end{enumerate} 
\end{lem}
\begin{proof}
It is easily seen that the condition $\varphi_\mathcal{M}(\phi^{-1}(N)) = \phi^{-1}(N)$ is equivalent to the condition $\varphi_{\mathcal{B}}(N) = N$. Then by commutativity of the diagram
\begin{equation*}
\begin{CD}
T^\perp_p \phi^{-1} (N) @>d\varphi_{\mathcal{M}}>> T^\perp_p \phi^{-1} (N)
\\
@V d\phi VV @V d\phi VV
\\
\ T^\perp_{\phi(p)} N\  @>d\varphi_{\mathcal{B}}>>\  T^\perp_{\phi(p)} N,
\end{CD}
\end{equation*}
the condition $d\varphi_{\mathcal{M}}(X) \neq X$ is equivalent to the condition $d \varphi_{\mathcal{B}}(d \phi(X)) \neq d \phi (X)$. This proves the lemma.
\end{proof}

\begin{proof}[Proof of Theorem $\ref{thm4}$]
(i) Take $a \in \pi^{-1}(N)$ and $w \in (T^\perp_a \pi^{-1}(N)) \backslash \{0\}$. Set $\eta := d \pi(w) \in T^\perp_{aK}N$, $N' := L_a^{-1}(N)$ and $\xi := dL_a^{-1}(\eta) \in (T^\perp_{eK} N') \backslash \{0\}$. Denote by $v \in (T_e^\perp \pi^{-1}(N')) \backslash \{0\}$ the horizontal lift of $\xi$. 
By commutativity of (\ref{commute1}) (ii)  we have $l_a (\pi^{-1}(N')) = \pi^{-1}(N)$ and $d l_a (v) = w$. Thus in order to show the existence of an isometry $\varphi_w$ with respect to $w$ it suffices to construct an isometry $\varphi_v$ with respect to $v$. Let $\varphi_{\eta}$ be an isometry with respect to $\eta$. Then an isometry $\varphi_{\xi}$ with respect to $\xi$ is defined  by
$
\varphi_{\xi} := L_a^{-1} \circ \varphi_{\eta} \circ L_a.
$
Now we define $\varphi_v$ as follows. Denote by $I(G/K)$  the group of isometries of $G/K$. 
From the assumption the map $L : G \rightarrow I(G/K)$, $a \mapsto L_a$ is a Lie group isomorphism onto the identity component $I_0(G/K)$ (\cite[Theorem 4.1 in Chapter V]{Hel01}). Since $I_0(G/K)$ is a normal subgroup of $I(G/K)$ an automorphism  $\varphi_v: G \rightarrow G$, $b \mapsto \varphi_v(b)$ is defined by
\begin{equation}\label{reflection3}
L_{\varphi_v(b)} := \varphi_{\xi} \circ L_b \circ \varphi_{\xi}^{-1}.
\end{equation}
Since the bi-invariant Riemannian metric on $G$ is induced from the Killing form of $\mathfrak{g}$ the automorphism $\varphi_v$ is an isometry of $G$. Moreover since
\begin{align*}
\varphi_\xi \circ \pi (b) 
&=
\varphi_\xi(bK)
=
\varphi_\xi \circ L_b (eK)
=
\varphi_\xi \circ L_b  \circ \varphi_\xi^{-1}(eK) \quad \text{and}
\\
\pi \circ \varphi_v(b)
&=
\varphi_v(b)K
=
L_{\varphi_v(b)}(eK)
=\varphi_\xi \circ L_b  \circ \varphi_\xi^{-1}(eK)
\end{align*}
hold for all $b \in G$ it follows from Lemma \ref{weaksubm} that $\varphi_v$ is an isometry with respect to $v$. This proves (i). 

(ii) Take $u \in \Phi_K^{-1}(N)$ and $X \in (T_u^\perp \Phi_K^{-1}(N)) \backslash \{0\}$. Also take $g \in P(G, G \times \{e\})$ so that $u = g * \hat{0}$. Set $a := \Phi(u) = g(0)$ and $\eta := d \Phi_K(X) \in (T^\perp_{aK} N) \backslash \{0\}$. Define $N'$, $\xi$, $v$ as in the above (i). Denote by $\hat{\xi} \in (T^\perp_{\hat{0}} \Phi_K^{-1}(N')) \backslash \{0\}$ the horizontal lift of $\xi$ with respect to the Riemannian submersion $\Phi_K: V_\mathfrak{g} \rightarrow G/K$. By commutativity of (\ref{commute3}) we have $ g*\Phi_K^{-1}(N') = \Phi^{-1}_{K}(N)$ and $d(g *)\hat{\xi} = X$. Thus in order to show the existence of an isometry $\varphi_X$ with respect to $X$ it suffices to construct an isometry $\varphi_{\hat{\xi}}$ with respect to $\hat{\xi}$. By the same way as in (i) we can define an isometry $\varphi_v$ with respect to $v \in (T^\perp_e \pi^{-1}(N')) \backslash\{0\}$. Moreover we define a linear orthogonal transformation $\varphi_{\hat{\xi}}$ of $V_\mathfrak{g}$ by 
\begin{equation} \label{reflection2}
\varphi_{\hat{\xi}} (u) := d \varphi_v \circ u, 
\quad 
u \in {V_\mathfrak{g}}.
\end{equation}
Since $\varphi_v$ is an automorphism of $G$ we have $\varphi_{\hat{\xi}}(g * \hat{0}) = (\varphi_v \circ g) * \hat{0}$ for all $g \in \mathcal{G}$. This together with \eqref{equiv4} implies that the following diagram commutes:
\begin{equation*}
\begin{CD}
V_\mathfrak{g} @>\varphi_{\hat{\xi}}>> V_\mathfrak{g}
\\
@V \Phi VV @V \Phi VV
\\
\ G \ @>\varphi_v >> \ G.
\end{CD}
\end{equation*}
Thus by Lemma \ref{weaksubm} $\varphi_{\hat{\xi}}$ is an isometry with respect to $\hat{\xi}$ and (ii) follows.\end{proof}

\begin{example}\label{example3}
Let $m,n \in \mathbb{Z}_{\geq 2}$. Set $\bar{M} : = S^{mn-1}(\sqrt{m}) \subset \mathbb{R}^{mn}$ and 
\begin{equation*}
M := 
\underbrace{S^{n-1} (1) \times \cdots \times S^{n-1}(1)}_{\text{$m$ times}}
\subset \bar{M}. 
\end{equation*}
Ikawa, Sakai and Tasaki \cite[Example 2.3]{IST09} showed that if $m = 2$ then $M$ is a weakly reflective submanifold of $\bar{M}$. Taketomi \cite[Proposition 3.1]{Tak18}  showed that if $m \geq 3$ then  $M$ is an arid submanifold of $\bar{M}$ and is not an austere submanifold (therefore not a weakly reflective submanifold) of $\bar{M}$. 

Set $G := SO(mn)$ and $K := SO(mn-1)$ so that $\bar{M} = G/K$. If $m = 2$ then $\Phi_K^{-1}(M)$ is a weakly reflective PF submanifold of $V_\mathfrak{g}$ by \cite[Theorem 8]{M19} and is not a totally geodesic PF submanifold of $V_\mathfrak{g}$ by \cite[Theorem 3]{M19}. If $m \geq 3$ then $\Phi_K^{-1}(M)$ is an arid PF submanifold of $V_\mathfrak{g}$ by Theorem \ref{thm4} and is not an austere PF submanifold (therefore not a weakly reflective PF submanifold) of $V_\mathfrak{g}$ by Theorem \ref{sphere2}. 
\end{example}

In general it is not clear that conversely the arid property of $\Phi_K^{-1}(N)$ implies the arid property of $N$ or not. 
However the next theorem shows that  under suitable assumptions the arid property of $\Phi_K^{-1}(N)$ is equivalent to that of $N$. To explain  this we now introduce some terminologies which are used in a context slightly wider than \cite{Tak18}. Let $M$ be a submanifold immersed in a finite dimensional Riemannian manifold $\bar{M}$. Denote by $I(\bar{M})$ the group of isometries of $\bar{M}$. For a closed subgroup $\mathfrak{G}$ of $I(\bar{M})$ we say that $M$ is \emph{$\mathfrak{G}$-arid} (resp.\ \emph{$\mathfrak{G}$-weakly reflective}) if for each $p \in M$ and each $\xi \in T^\perp_p M \backslash \{0\}$ there exists an isometry $\varphi_\xi$ (resp.\ reflection $\nu_\xi$) with respect to $\xi$ satisfying $\varphi_\xi \in \mathfrak{G}_p$, where $\mathfrak{G}_p$ denotes the isotropy subgroup of $\mathfrak{G}$ at $p$.  Clearly $\mathfrak{G}$-weakly reflective submanifols are $\mathfrak{G}$-arid submanifolds. If $\mathfrak{G} = I(\bar{M})$ then  ``$\mathfrak{G}$-arid'' (resp.\ ``$\mathfrak{G}$-weakly reflective'') is nothing but ``arid'' (resp.\ ``weakly reflective'').
The same concepts and relation are also valid for PF submanifolds in Hilbert spaces.

From now on, as in Section \ref{Sec2},  we denote by $G$ a connected compact Lie group equipped with a bi-invariant Riemannian  metric, $K$ a closed subgroup of $G$ and $G/K$ the compact normal homogeneous space. Recall that  the Hilbert Lie group $\mathcal{G} := H^1([0,1], G)$ acts on $V_\mathfrak{g}$ via the gauge transformations \eqref{gaugetransf}. We denote by $\mathcal{G}_u$ the isotropy subgroup of $\mathcal{G}$ at $u \in V_\mathfrak{g}$. If $u = \hat{0}$ then $\mathcal{G}_{\hat{0}}$ is the set of constant paths $\hat{G} := \{\hat{b} \in G \mid b \in G\}$. Thus if $u = g * \hat{0}$ for some $g \in \mathcal{G}$ then $\mathcal{G}_u = g \hat{G} g^{-1}$. 
Recall also that $G \times G$ acts on $G$ by the formula \eqref{Haction}. We denote by $(G \times G)_a = (a,e) \Delta G (a,e)^{-1}$ the isotropy subgroup of $G \times G$ at $a \in G$, where $\Delta  G := \{(b,b)\mid b\in G\}$. Finally we recall the $G$-action on $G/K$ defined by $b \cdot (aK) := (ba)K$ for $a, b \in G$. We denote by $G_{aK} = aKa^{-1}$ the isotropy subgroup of $G$ at $aK \in G/K$.
 
\begin{thm}\label{thm2}
\begin{enumerate}\ 
\item[(i)] Let $N$ be a closed submanifold of a compact Lie group $G$   equipped with a bi-invariant Riemannian metric. Then the following are equivalent: 
\begin{enumerate}
\item[(a)] $N$ is a $(G \times G)$-arid submanifold of $G$.
\item[(b)] $\Phi^{-1}(N)$ is a $\mathcal{G}$-arid PF submanifold of $V_\mathfrak{g}$.
\end{enumerate}
\item[(ii)] Let $N$ be a closed submanifold of a compact normal homogeneous space $G/K$. Then the  following are equivalent:
\begin{enumerate}
\item[(a)] $N$ is a $G$-arid submanifold of $G/K$. 
\item[(b)] $\pi^{-1}(N)$ is a $(G \times K)$-arid submanifold of $G$.
\item[(c)] $\Phi_K^{-1}(N)$ is a $P(G, G \times K)$-arid PF submanifold of $V_\mathfrak{g}$.
\end{enumerate}
\end{enumerate}
\end{thm}
\begin{proof}
(i) ``(a) $\Rightarrow$ (b)'': Take $u \in \Phi^{-1}(N)$ and $X \in (T_u^\perp \Phi^{-1}(N)) \backslash \{0\}$. Also take $g \in P(G, G \times \{e\})$ so that $u = g * \hat{0}$. Set $a := \Phi(u) = g(0)$, $\eta := d \Phi(X) \in (T_{a} ^\perp N) \backslash \{0\}$, $N' := a^{-1}N$ and $\xi := dL_a^{-1}( \eta) \in (T_e^\perp N') \backslash \{0\}$. Denote by $\hat{\xi} \in (T^\perp_{\hat{0}} \Phi^{-1}(N')) \backslash \{0\}$ the horizontal lift of $\xi$. By commutativity of (\ref{commute1}) (i) we have $g * (\Phi^{-1}(N')) = \Phi^{-1}(N)$ and $(d g*) \hat{\xi} = X$. Thus in order to show the existence of an isometry $\varphi_X$ with respect to $X$ satisfying $\varphi_X \in \mathcal{G}_u$ it suffices to construct an isometry $\varphi_{\hat{\xi}}$ with respect to $\hat{\xi}$ satisfying $\varphi_{\hat{\xi}} \in \mathcal{G}_{\hat{0}}$. Let $\varphi_{\eta}$ be an isometry with respect to $\eta$ which is given by $\varphi_\eta(c) = b' c b^{-1}$ for some $(b', b) \in (G \times G)_a$. Then an isometry $\varphi_{\xi}$ with respect to $\xi$ is defined  by 
$
\varphi_{\xi} := (a, e)^{-1} \circ \varphi_{\eta} \circ (a,e)
$,
that is, $\varphi_{\xi}(c) := b c b^{-1}$ for $c \in G$. Now we define a linear orthogonal transformation $\varphi_{\hat{\xi}}$ of $V_\mathfrak{g}$  by
\begin{equation*}
\varphi_{\hat{\xi}} (u) 
:= 
d \varphi_\xi \circ u
=
b u b^{-1}
=
\hat{b} * u,
\quad 
u \in {V_\mathfrak{g}}.
\end{equation*}
Clearly $\varphi_{\hat{\xi}} \in \mathcal{G}_{\hat{0}}$. Moreover by (\ref{equiv4}) the following diagram commutes:
\begin{equation*}
\begin{CD}
V_\mathfrak{g} @>\varphi_{\hat{\xi}}>> V_\mathfrak{g}
\\
@V \Phi VV @V \Phi VV
\\
\ G\  @>\varphi_\xi >> \ G.
\end{CD}
\end{equation*}
From Lemma \ref{weaksubm} $\varphi_{\hat{\xi}}$ is an isometry with respect to $\hat{\xi}$ and (b) follows.

(i) ``(b) $\Rightarrow$ (a)'': Take $a \in N$ and $\eta \in (T_a^\perp N)\backslash \{0\}$. Fix $u \in \Phi^{-1}(a)$. Denote by $X \in (T^\perp_u \Phi^{-1}(N)) \backslash \{0\}$ the horizontal lift of $\eta$. Take $g \in P(G, G \times \{e\})$ so that $g*\hat{0} = u$ and define  $N'$, $\xi$, $\hat{\xi}$ as in the above (i).  Let $\varphi_X$ be an isometry with respect to $X$ satisfying $\varphi_X \in \mathcal{G}_u$. Then an isometry $\varphi_{\hat{\xi}}$ with respect to $\hat{\xi} \in (T^\perp_{\hat{0}} \Phi^{-1}(N'))\backslash \{0\}$ is defined by $\varphi_{\hat{\xi}} := (g*)^{-1} \circ \varphi_X \circ (g*)$. By definition $\varphi_\xi \in \mathcal{G}_{\hat{0}}$ and thus there exists $b \in G$ such that $\varphi_{\hat{\xi}}(u) = b u b^{-1}$. Hence defining an isometry $\varphi_\xi$ of $G$ by $\varphi_\xi(c) :=bcb^{-1}$ for $c \in G$ it follows from Lemma \ref{weaksubm} that $\varphi_\xi$ is an isometry with respect to $\xi$ satisfying $\varphi_\xi \in (G \times G)_e$. Therefore an isometry $\varphi_\eta$  with respect to $\eta$ is defined by $\varphi_{\eta} := l_a \circ \varphi_\xi \circ l_a^{-1}$ so that $\varphi_\eta \in (G \times G)_a$. This shows (a).

(ii) ``(a) $\Rightarrow$ (b)'':  Take $a \in \pi^{-1}(N)$ and $w \in (T_a^\perp \pi^{-1}(N)) \backslash \{0\}$. Set $\eta:= d \pi (w) \in (T^\perp_{aK} N) \backslash \{0\}$, $N' := L_a^{-1}(N)$ and $\xi := d L_a^{-1}(\eta) \in (T^\perp_{eK}N') \backslash \{0\}$. Denote by $v \in (T_e^\perp \pi^{-1}(N')) \backslash \{0\}$ the horizontal lift of $\xi$. By commutativity of (\ref{commute1}) (ii), we have $l_a (\pi^{-1}(N')) = \pi^{-1}(N)$ and $d l_a (v) = w$. Thus in order to show the existence of an isometry $\varphi_w$ with respect to $w$ satisfying $\varphi_w \in (G \times K)_a = (a,e) \Delta K (a,e)^{-1}$ it suffices  to construct an isometry $\varphi_v$ with respect to $v$ satisfying $\varphi_v \in (G \times K)_e = \Delta K$. Let $\varphi_\eta$ be an isometry with respect to $\eta$ which is given by $\varphi_\eta(cK)= (bc)K$ for some $b \in G_{aK}$. Then there exists $k \in K$ such that $b = aka^{-1}$. Thus an isometry $\varphi_{\xi}$ with respect to $\xi$ is defined by $\varphi_\xi := L_a^{-1} \circ \varphi_\xi \circ L_a$, that is, $\varphi_\xi = L_k$. Define an isometry $\varphi_v$ of $G$ by
\begin{equation*}
\varphi_{v}(c) :=  k c k^{-1}
, \quad 
c \in G.
\end{equation*}
Clearly $\varphi_v \in (G \times K)_e$. Moreover the following diagram commutes:
\begin{equation*}
\begin{CD}
G @>\varphi_v>> G
\\
@V \pi VV @V \pi VV
\\
\ G/K\  @>\varphi_{\xi}>> \ G/K.
\end{CD}
\end{equation*}
Thus by Lemma \ref{weaksubm} $\varphi_v$ is an isometry with respect to $v$ and (b) follows.

(ii) ``(b) $\Rightarrow$ (a)'': Take $aK \in N$ and $\eta \in T_{aK}^\perp N$. Denote by $w \in T^\perp_a \pi^{-1}(N)$ the horizontal lift of $\eta$. Define $N'$, $\xi$, $v$ as above. Let $\varphi_{w}$ be an isometry with respect to $w$  satisfying $\varphi_w \in (G \times K)_a$. Then an isometry with respect to 
$v$ is defined by $\varphi_v := l_a^{-1} \circ \varphi_w \circ l_a$ so that $\varphi_v \in (G \times K)_e$. Then there exists $k \in K$ such that $\varphi_v(c) = kck^{-1}$. Thus defining an isometry $\varphi_\xi$ of $G/K$ by 
$\varphi_\xi := L_k$ it follows from Lemma \ref{weaksubm} that $\varphi_\xi$ is an isometry with respect to $\xi$ satisfying $\varphi_\xi \in G_{eK}$. Hence an isometry $\varphi_\eta$ with respect to $\eta$ is defined by $\varphi_\eta := l_a \circ \varphi_\xi \circ l_a^{-1}$ so that $\varphi_\eta \in G_{aK}$. This shows (b).

Using the fact
$g P(G, G \times K)_{\hat{0}}g^{-1} = P(G, G \times K)_{g * \hat{0}}$
for $g \in P(G, G \times \{e\})$
the equivalence of (b) and (c) of (ii) follows by the similar arguments to (i).
\end{proof}

\begin{cor}\label{cor} 
Let $G$, $G/K$ be as in Theorem \ref{thm2}. 
\begin{enumerate}
\item[(i)] Let $H$ be a closed subgroup of $G \times G$. Then the following are equivalent:
\begin{enumerate}
\item[(a)] an orbit $H \cdot a$ through $a \in G$ is an $H$-arid submanifold of $G$,
\item[(b)] an orbit $P(G,H)* u$ through $u \in \Phi(a)$ is a $P(G,H)$-arid PF submanifold of $V_\mathfrak{g}$.
\end{enumerate}
\item[(ii)] Let $K'$ be a closed subgroup of $G$. Then the following are equivalent:
\begin{enumerate}
\item[(a)] an orbit $K' \cdot aK$ through $aK \in G/K$ is a $K'$-arid submanifold of $G/K$,

\item[(b)] an orbit $(K' \times K) \cdot a$ through $a \in G$ is a $(K' \times K)$-arid submanifold of $G$,
\item[(c)] an orbit $P(G,K' \times K) * u$ through $u \in \Phi^{-1}(a)$ is a $P(G,K' \times K)$-arid PF submanifold of $V_\mathfrak{g}$.
\end{enumerate}
\end{enumerate}
\end{cor}
\begin{proof}
(i) Set $H' := (a,e)^{-1}H (a,e)$. Then we have $H \cdot a = l_a(H' \cdot e)$.
Take $g \in P(G, G \times \{e\})$ so that $g* \hat{0} =u$. From  \eqref{iimage} and \eqref{commute1} (i) it follows that 
$P(G,H) * u = g*(P(G, H') *\hat{0})$.
From this fact we can assume without loss of generality that $a = e$.
Moreover by homogeneity it suffices to consider normal vectors only at $e \in G$ or $\hat{0} \in V_\mathfrak{g}$. By similar arguments as in Theorem \ref{thm2} (i) our claim follows. (ii) Similarly we can reduce the case $a =e$ and the assertion follows by similar arguments as in Theorem \ref{thm2} (ii).
\end{proof}

\begin{example}\label{example1}
Let $(U, L)$ be a compact Riemannian symmetric pair where $L$ connected. Denote by $\mathfrak{u} = \mathfrak{l} + \mathfrak{p}$ the canonical decomposition, by $\operatorname{Ad}: L \rightarrow SO(\mathfrak{p})$ the isotropy representation and by $S$ the standard sphere in $\mathfrak{p}$.

Let us first show that there are  examples of $\operatorname{Ad}(L)$-orbits which are $\operatorname{Ad}(L)$-arid submanifolds in $S$. By Taketomi's result (\cite[Proposition 4.4]{Tak18}) an orbit $N := \operatorname{Ad}(L) w$ through $w \in S$ is an $\operatorname{Ad}(L)$-arid submanifold of $S$ if and only if $N$ is an isolated orbit of the $\operatorname{Ad}(L)$-action on $S$. One can find such isolated orbits by considering the fundamental Weyl Chamber. Fix a maximal abelian subspace $\mathfrak{a}$ in $\mathfrak{p}$ and denote by $F$ the fundamental system of the restricted root system with respect to $\mathfrak{a}$. The fundamental Weyl chamber is defined by 
\begin{equation*}
C := \{w \in \mathfrak{a} \mid \forall \alpha \in F , \ \alpha(w)>0 \}
\end{equation*}
with closure 
\begin{equation*}
\bar{C} := \{w \in \mathfrak{a} \mid \forall \alpha \in F , \ \alpha(w) \geq 0 \}.
\end{equation*}
It is known (\cite[Lemma 1.2]{HKT00}) that $\bar{C}$ is decomposed by 
\begin{equation*}
\bar{C} = \coprod_{\text{$\Delta$: subset of $F$}} C^\Delta
\qquad \text{: disjoint union},
\end{equation*}
\begin{equation*}
C^\Delta := \{w \in \mathfrak{a} \mid \forall \alpha \in \Delta, \ \alpha(w) >0 \ \ \text{and}\ \ \forall \beta \in F \backslash \Delta , \ \beta(w) =0\}.
\end{equation*}
If a subset $\Delta$ consists of only one element then $\dim C^\Delta = 1$ and thus the intersection $S \cap C^\Delta$ consists of only one point, which implies that in this case the orbit $\operatorname{Ad}(L)w$ through $w \in C^\Delta$ is isolated. In this way we can obtain examples of $\operatorname{Ad}(L)$-arid submanifolds in the standard sphere $S$. Notice that from the classification result of austere $\operatorname{Ad}(L)$-orbits (\cite[Theorem 5.1]{IST09}), in particular we can choose $\operatorname{Ad}(L)$-arid orbits which are \emph{not} austere.

Applying Corollary \ref{cor} (ii) to such examples we obtain the orbit \begin{equation*}
P(SO(\mathfrak{p}), \operatorname{Ad}(L) \times SO(\mathfrak{p})_w) * \hat{0}\end{equation*}
which is an $P(SO(\mathfrak{p}), \operatorname{Ad}(L) \times SO(\mathfrak{p})_w)$-arid PF submanifolds in the Hilbert space $V_{\mathfrak{o}(\mathfrak{p})}$. Moreover by  
Theorem \ref{sphere2} such an arid PF submanifold is \emph{not} an austere (therefore \emph{not} a weakly reflective) PF submanifold in $V_\mathfrak{o(p)}$.
\end{example}

Compared to Corollary \ref{cor} (i) the following proposition covers only $H$-orbits through $e \in G$. However the isometry $\varphi_\xi$ with respect to each normal vector $\xi$ at $e \in G$ need \emph{not} belong to the isotropy subgroup $H_e$ at $e \in G$.

\begin{prop}\label{type2.2}
Let $G$ be a connected compact Lie group with a bi-invariant Riemannian metric and $H$ be a closed subgroup of $G \times G$. Suppose that the orbit $H \cdot e$ through $e \in G$ is an arid submanifold of $G$ such that for each $\xi \in (T^\perp_e (H \cdot e)) \backslash \{0\}$ an isometry $\varphi_\xi$ with respect to $\xi$ is an automorphism of $G$. Then the orbit  $P(G,H) * \hat{0}$ through $\hat{0} \in V_\mathfrak{g}$ is an  arid PF submanifold of $V_\mathfrak{g}$. 
\end{prop}
\begin{proof}
Let $\varphi_\xi$ be an isometry with respect to $\xi \in (T^\perp_e (H \cdot e)) \backslash \{0\}$ which is an automorphism of $G$. 
Then an isometry $\varphi_{\hat{\xi}}$ with respect to $\hat{\xi} \in (T^\perp_{\hat{0}} P(G,H) * \hat{0}) \backslash \{0\}$ is defined similarly to  (\ref{reflection2}). By homogeneity of $P(G, H) * \hat{0}$ our claim follows.
\end{proof}

Now we see an example of an arid submanifold $H \cdot e$ satisfying the condition in Proposition \ref{type2.2}. Although the following $H \cdot e$ can be shown to be arid by applying Theorem \ref{thm4} (i) to Taketomi's example \cite[Proposition 3.1]{Tak18}, here we give a direct proof in order to see an isometry with respect to each normal vector explicitly.
\begin{example}\label{example2}
Set $G := SO(9)$. Denote by $E$ the $3 \times 3$ unit matrix. Define $Q \in G$ by 
\begin{equation*}
Q := 
\frac{1}{\sqrt{6}}
\left[
\begin{array}{ccc}
\sqrt{2} E & \sqrt{3} E & E
\\
\sqrt{2} E & -\sqrt{3} E & E
\\
\sqrt{2} E & 0 & -2 E
\end{array}
\right].
\end{equation*}
Set $K := Q(\{1\} \times SO(8))Q^{-1}$, $K' := SO(3) \times SO(3) \times SO(3)$ and $H := K' \times K$. Then the tangent space of the orbit $H \cdot e$ is given by 
\begin{equation*}
T_{e}(H \cdot e)
=
\mathfrak{k}' + \mathfrak{k}
=
\mathfrak{k}' +Q(0 \oplus \mathfrak{o}(8)) Q^{-1}
=
Q( Q^{-1}\mathfrak{k}'Q + ( 0 \oplus \mathfrak{o}(8))) Q^{-1}.
\end{equation*}
Then it follows from the direct computations that each $X \in Q^{-1} T^\perp_e (H \cdot e) Q$ is written by
\begin{equation*}
X = \left[
\begin{array}{ccc}
0& S&T
\\
-S& 0& 0
\\
-T& 0& 0
\end{array}
\right],
\end{equation*}
\begin{equation*}
S := 
\left[
\begin{array}{ccc}
s &0 & 0 
\\ 
0& 0& 0
\\
0& 0&0
\end{array}
\right]
, \quad
T := 
\left[
\begin{array}{ccc}
t &0& 0 
\\ 
0&0& 0
\\
0& 0& 0
\end{array}
\right]
, \quad
s, t \in \mathbb{R}.
\end{equation*}
The calculation of $Q X Q^{-1}$ shows that each $Y \in T^\perp_e (H \cdot e)$ is written by
\begin{equation*}
Y =\left[
\begin{array}{ccc}
0 & U & V
\\
-U & 0 & W
\\
-V&-W&0
\end{array}
\right],
\end{equation*}
\begin{align*}
U :=
\left[
\begin{array}{ccc}
-2x&0&0
\\
0&0&0
\\
0&0&0
\end{array}
\right]
, \quad 
V :=
\left[
\begin{array}{ccc}
-x-y&0&0
\\
0&0&0
\\
0&0&0
\end{array}
\right]
,\quad 
W :=
\left[
\begin{array}{ccc}
x-y&0&0
\\
0&0&0
\\
0&0&0
\end{array}
\right]
\end{align*}
for $x,y \in \mathbb{R}$. For each $(i,j)  \in \{(1,2), (1,3), (2,3)\}$ we define $P_{i,j} \in O(9)$ by
\begin{equation*}
P_{1,2}
:=
\left[
\begin{array}{ccc}
0 & E & 0
\\
E&0 &0
\\
0&0& E
\end{array}
\right]
, \quad
P_{1,3}
:=
\left[
\begin{array}{ccc}
0 & 0 & E
\\
0&E &0
\\
E&0& 0
\end{array}
\right]
, \quad
P_{2,3}
:=
\left[
\begin{array}{ccc}
E & 0 & 0
\\
0&0 &E
\\
0&E& 0
\end{array}
\right]
\end{equation*} 
and define an automorphism $\varphi_{ij}$ of $G$ by 
\begin{equation*}
\varphi_{ij} (A) := P_{ij} A P_{ij}^{-1}
, \quad A \in G.
\end{equation*}
Then for each $\xi \in T^\perp_e (H \cdot e)$ there exists $(i,j)$ such that $\varphi_{ij}$ is an isometry  with respect to $\xi$. Thus $H \cdot e$ is an arid submanifold of $G$. Since $\varphi_{ij}$ is an automorphism of $G$ it follows from Proposition \ref{type2.2} that the orbit $P(G,H)* \hat{0}$ is an arid PF submanifold of $V_\mathfrak{g}$. 
Notice that we can not apply Corollary \ref{cor} (i) to this example since $\varphi_{ij}$ is \emph{not} an inner automorphism of $G$ and thus not belong to the isotropy subgroup $H_e$ at $e \in G$.
\end{example}

\section{Open problems}\label{Kac}
Recall that an isometric action of a compact Lie group on a Riemannian manifold $M$ is called \emph{polar} if there exists a closed connected submanifold $S$ of $M$ which meets every orbit orthogonally. If $S$ is flat in the induced metric then such an action is called \emph{hyperpolar} (\cite{HPTT95}). Similarly an isometric PF action of a Hilbert Lie group on a Hilbert space $V$ is called \emph{hyperpolar} (\cite{HPTT95}) if there exists a closed affine subspace $S$ of $V$ which meets every orbit orthogonally.  It was shown (\cite{Ter89}) that the  $P(G,H)$-action is hyperpolar if the $H$-action on $G$ is hyperpolar. Hyperpolar actions on irreducible Riemannian symmetric spaces of compact type were classified by Kollross \cite{Kol02}.

In this paper and the previous paper \cite{M19} we have seen examples of
minimal PF submanifolds with certain symmetries, such as weakly reflective PF submanifolds, austere PF submanifolds and arid PF submanifolds. At present all such examples are obtained as orbits of  hyperpolar $P(G,H)$-actions. Then it is interesting to study the following problem:
\begin{problem} 
Determine minimal orbits in a hyperpolar $P(G,H)$-action and classify the symmetric properties they have.
\end{problem}

It is noted (\cite{HPTT95}) that a  hyperpolar $P(G,H)$-action can be thought of the isotropy representation of an affine Kac-Moody symmetric space, which is an infinite dimensional version of a symmetric space proposed by Terng \cite{Ter95} and established by Heintze \cite{Hei06}, Popescu \cite{Pop05} and Freyn \cite{Fre11}. In the finite dimensional case minimal orbits of the isotropy representation of symmetric spaces have been systematically studied  (e.g.\ \cite{HTST00}, \cite{IST09}). It may be  interesting to ask whether there are similar properties for minimal orbits in the isotropy representations of Kac-Moody symmetric spaces.

In order to study the above problem it is noted (\cite{KT93}, \cite{HLO06}) that a $P(G,H)$-orbit is minimal if and only if the $H$-orbit is minimal. This shows that the determination of minimal $P(G,H)$-orbits can be reduced to a finite dimensional problem. However the classification of symmetric properties seems not easy. As we have seen in Section \ref{austerevia} the austere property via the parallel transport map is not clear except for the spherical case. Moreover even if we classify all austere orbits in $H$- or $P(G,H)$-actions it seems very difficult in general to classify all weakly reflective orbits; it is very hard to assert one austere orbit is \emph{not} weakly reflective.

In connection with such a classification problem the author is now giving attention to the structure of weakly reflective submanifolds. Let $M$ be a weakly reflective submanifold of a  Riemannian manifold $\bar{M}$. 
Denote by $I(\bar{M})$ the group of isometries of  $\bar{M}$ and by $I_0(\bar{M})$ its  identity component. 
There are at least two kinds of weakly reflective submanifolds:
\begin{enumerate}
\item[(a)] for each $p \in M$ and each $\xi \in T^\perp_p M$ there exists $\nu_\xi \in I_0(\bar{M})$ (: \emph{identity component}) such that $\nu_\xi(p)= p$, $d \nu_\xi(\xi) = - \xi$ and $\nu_\xi(M) = M$. 
\item[(b)] for each $p \in M$ there exists an \emph{involutive} isometry $\nu_p \in I(\bar{M})$ which is \emph{independent} of the choice of $\xi \in T^\perp_p M$ such that $\nu_p(p)= p$, $d \nu_p(\xi) = - \xi$ and $\nu_p(M) =M$. 
\end{enumerate}
For example, consider an isometric action of cohomogeneoity one on $\bar{M}$. It is known (\cite{Pod97}, \cite{IST09}) that in this case any singular orbit $M$ is a weakly reflective submanifold of $\bar{M}$. More precisely if $M$ has one codimension in $\bar{M}$ then $M$ satisfies the condition (b) 
and otherwise $M$ satisfies the condition (a). Note that there exist examples of weakly reflective submanifolds which satisfy both conditions (a) and (b): 
a weakly reflective submanifold $M$ reviewed in Example \ref{example3} (the case $m = 2$) satisfies (b) for all $n \in \mathbb{Z}_{\geq 2}$ and in particular if $n$ is even then it also satisfies (a). Although examples by Ohno \cite{Ohno16} and by Kimura-Mashimo \cite{KM19} satisfy the condition (a), Enoyoshi's example (\cite{Eno18}) does not satisfy (a) but (b). It should be also noted (\cite{M19}) that under suitable assumptions if $M$ satisfies the condition (b) then its inverse image under the parallel transport map is a weakly reflective PF submanifold satisfying the condition (b). It can be a problem to consider whether weakly reflective submanifolds which does not satisfy both conditions (a) and (b) exist or not. Also for arid submanifolds similar types may work and might be useful to study their structure.

\section*{Acknoledgements}
The author would like to thank  Professor Yoshihiro Ohnita for many  discussions and invaluable  suggestions. The author is also grateful to Professors Takashi Sakai and Hiroshi Tamaru for introducing the concept of arid submanifolds to me and suggesting me to study them in the Hilbert space. Special thanks are also due to Professors Ernst Heintze and Jost-Hinrich Eschenburg for their interests in my work and having many discussions with me, and to Professor Peter Quast for his kind hospitality during the author's visit to the University of Augsburg.



\begin{thebibliography}{9}
\bibitem{Bry91} R. L. Bryant, \emph{Some remarks on the geometry of austere manifolds}, Bol. Soc. Brasil. Mat. (N.S.) \textbf{21} (1991), no. 2, 133-157.


\bibitem{DF01} M. Dajczer, L. A. Florit, \emph{A class of austere submanifolds}, Illinois J. Math. \textbf{45} (2001), no. 3, 735-755. 

\bibitem{Eno18} K. Enoyoshi, \emph{Principal curvatures of homogeneous hypersurfaces in a Grassmann manifold $\widetilde {\operatorname{Gr}}_3 (\operatorname{Im}\mathbb{O})$ by the $G_2$-action}, to appear in Tokyo J. Math.

\bibitem{Fre11} W. Freyn, \emph{Affine Kac-Moody symmetric spaces}, arXiv:1109.2837 (2011).

\bibitem{HL82} R. Harvey and H. B. Lawson, Jr., \emph{Calibrated geometries}, Acta Math., \textbf{148} (1982), 47-157.

\bibitem{Hei06} E. Heintze, \emph{Toward symmetric spaces of affine Kac-Moody type}, Int. J. Geom. Methods Mod. Phys. \textbf{3} (2006), no. 5-6, 881-898.

\bibitem{HLO06} E. Heintze, X. Liu, C. Olmos,  \emph{Isoparametric submanifolds and a Chevalley-type restriction theorem}, Integrable systems, geometry, and topology, 151-190, AMS/IP Stud. Adv. Math., \textbf{36}, Amer. Math. Soc., Providence, RI, 2006.

\bibitem{HPTT95} E. Heintze, R. Palais, C.-L. Terng, G. Thorbergsson, \emph{Hyperpolar actions on symmetric spaces}, Geometry, topology, \& physics, 214–245, Conf. Proc. Lecture Notes Geom. Topology, IV, Int. Press, Cambridge, MA, 1995.

\bibitem{Hel01} S. Helgason, \emph{Differential Geometry, Lie groups, and Symmetric Spaces}, Corrected reprint of the 1978 original. Graduate Studies in Mathematics, \textbf{34}. American Mathematical Society, Providence, RI, 2001.

\bibitem{HKT00} D. Hirohashi, T.  Kanno, H. Tasaki, \emph{Area-minimizing of the cone over symmetric R-spaces}, Tsukuba J. Math. \textbf{24} (2000), no. 1, 171-188. 

\bibitem{HTST00} D. Hirohashi, H. Tasaki, H. Song, R. Takagi, \emph{Minimal orbits of the isotropy groups of symmetric spaces of compact type}, Differential Geom. Appl. \textbf{13} (2000), no. 2, 167-177.

\bibitem{IST09}  O. Ikawa,  T. Sakai, H. Tasaki, \emph{Weakly reflective submanifolds and austere submanifolds}. J. Math. Soc. Japan \textbf{61} (2009), no. 2, 437-481.


\bibitem{II10} M. Ionel, T. Ivey, \emph{Austere submanifolds of dimension four: examples and maximal types}, Illinois J. Math. \textbf{54} (2010), no. 2, 713-746.


\bibitem{KM19} T. Kimura, K. Mashimo, \emph{Classification of Cartan embeddings which are austere submanifolds}, a preprint.

\bibitem{KT93} C. King, C.-L. Terng, \emph{Minimal submanifolds in path space}, Global analysis in modern mathematics, 253-281, Publish or Perish, Houston, TX, 1993.

\bibitem{Koi02} N. Koike, \emph{On proper Fredholm submanifolds in a Hilbert space arising from submanifolds in a symmetric space}, Japan. J. Math. (N.S.) \textbf{28} (2002), no. 1, 61-80.

\bibitem{Kol02} A. Kollross, \emph{A classification of hyperpolar and cohomogeneity one actions}, Trans. Amer. Math. Soc. \textbf{354} (2002), no. 2, 571-612.


\bibitem{Loo69II} O. Loos, \emph{Symmetric Spaces. II: Compact Spaces and Classification}. W. A. Benjamin, Inc., New York-Amsterdam 1969.

\bibitem{Leu73} Dominic S. P. Leung,  \emph{The reflection principle for minimal submanifolds of Riemannian symmetric spaces},  J. Differential Geom. \textbf{8} (1973), 153-160. 

\bibitem{Leu74} Dominic S. P. Leung, \emph{On the classification of reflective submanifolds of Riemannian symmetric spaces},  Indiana Univ. Math. J. \textbf{24} (1974/75), 327-339.

\bibitem{GR07} I. S. Gradshteyn,  I. M. Ryzhik, \emph{Table of Integrals, Series, and Products}, Translated from the Russian. Translation edited and with a preface by Alan Jeffrey and Daniel Zwillinger. Seventh edition. Elsevier/Academic Press, Amsterdam, 2007. 


\bibitem{M19} M. Morimoto, \emph{On weakly reflective PF submanifolds in Hilbert spaces}, to appear in Tokyo J. Math.


\bibitem{Ohno16} S. Ohno,  \emph{A sufficient condition for orbits of Hermann actions to be weakly reflective}, Tokyo J. Math. \textbf{39} (2016), no. 2, 537-564.

\bibitem{Pal63} R. S. Palais, \emph{Morse theory on Hilbert manifolds}, Topology 2 (1963), 299-340.

\bibitem{PT88} R. S. Palais, C.-L. Terng, \emph{Critical Point Theory and Submanifold Geometry}, Lecture Notes in Math., vol 1353, Springer-Verlag, Berlin and New York, 1988.


\bibitem{Pop05} B. Popescu, \emph{Infinite dimensional symmetric spaces}, Thesis, University of Augsburg (2005).

\bibitem{Pod97} F. Podest\`a, \emph{Some remarks on austere submanifolds}, Boll. Un. Mat. Ital. B (7) \textbf{11} (1997), no. 2, suppl., 157-160.

\bibitem{Sma64} S. Smale, \emph{Morse theory and a non-linear generalization of the Dirichlet problem}, Ann. of Math. (2)\textbf{80} (1964), 382-396. 

\bibitem{Tak18} Y. Taketomi, \emph{On a Riemannian submanifold whose slice representation has no nonzero fixed points}, Hiroshima Math. J. \textbf{48} (2018), no. 1, 1-20.


\bibitem{Ter89} C.-L.  Terng,  \emph{Proper Fredholm submanifolds of Hilbert space.} J. Differential Geom. \textbf{29} (1989), no. 1, 9-47.

\bibitem{Ter95} C.-L. Terng, \emph{Polar actions on Hilbert space}. J. Geom. Anal. \textbf{5} (1995), no. 1, 129-150. 

\bibitem{TT95} C.-L. Terng, G. Thorbergsson, \emph{Submanifold geometry in symmetric spaces.} J. Differential Geom. \textbf{42} (1995), no. 3, 665-718.
\end{thebibliography}
\end{document}